\documentclass[10 pt]{amsart}

\usepackage[latin1]{inputenc}
\usepackage{enumerate}
\usepackage[T1]{fontenc}
\usepackage[english]{babel}
\usepackage{amsmath}
\usepackage{verbatim}
\usepackage{amsfonts}
\usepackage[all]{xy}
\usepackage{amssymb}
\usepackage{tikz,pgffor}
\usetikzlibrary{arrows,matrix}
\usepackage{showlabels}
\usepackage{hyperref} 
\usepackage{mathrsfs}
\usepackage{color}
\usepackage{url}

\usepackage[margin=1.25in]{geometry}

\newcommand{\incl}[1][r]{\ar@<-0.2pc>@{^(-}[#1] \ar@<+0.2pc>@{-}[#1]}

\newcommand{\Cl}{\operatorname{Cl}}
\newcommand{\Pic}{\operatorname{Pic}}
\newcommand{\Aut}{\operatorname{Aut}}
\newcommand{\Spec}{\operatorname{Spec}}

\newcommand{\Bl}{\mathrm{Bl}}
\newcommand{\SL}{\mathrm{SL}}

\newcommand{\codim}{\mathrm{codim}}

\newcommand{\X}{\mathcal{X}}
\newcommand{\U}{\mathcal{U}}

\newcommand{\Y}{\mathcal{Y}}

\renewcommand{\O}{\mathcal{O}}

\newcommand{\Gm}{\mathbb{G}_m}
\newcommand{\A}{\mathbb{A}}

\newcommand{\T}{\mathbb{T}}
\renewcommand{\P}{\mathbb{P}}
\newcommand{\Z}{\mathcal{Z}}

\theoremstyle{plain}
\newtheorem{theorem}[subsection]{Theorem}
\newtheorem{lemma}[subsection]{Lemma}
\newtheorem{corollary}[subsection]{Corollary}
\newtheorem{proposition}[subsection]{Proposition}
\newtheorem*{theorem*}{Theorem} %%

\theoremstyle{definition}
\newtheorem{example}[subsection]{Example}
\newtheorem{remark}[subsection]{Remark}
\newtheorem{definition}[subsection]{Definition}
\newtheorem{conjecture}[subsection]{Conjecture}
\newtheorem*{conjecture*}{Conjecture \ref{conj}} %%

\title[Horospherical stacks]{Horospherical stacks}
\date{\today}

\author{Ariyan Javanpeykar}
\address{Institut f\"ur Mathematik\\
 Johannes Gutenberg-Universit\"at Mainz\\
    Staudingerweg 9 \\
     55099 Mainz, Germany}
\email{peykar@uni-mainz.de}

\author{Kevin Langlois}
\address{Mathematisches Institut\\ 
   Heinrich Heine Universit\"{a}t\\
   40225 D\"{u}sseldorf, Germany}
\email{langlois.kevin18@gmail.com}

\author{Ronan Terpereau}
\address{ Institut de Math\'ematiques de Bourgogne - UMR 5584 du CNRS \\
Univ. Bourgogne Franche-Comt\'e\\
  9 avenue Alain Savary \\
  BP 47870 - 21078 DIJON Cedex, France}
\email{ronan.terpereau@u-bourgogne.fr}


  %Toric varieties 

\begin{document}

\begin{abstract}  
We prove structure theorems for algebraic stacks with a reductive group action and a dense open substack isomorphic to a horospherical homogeneous space, and thereby obtain new examples of algebraic stacks which are global quotient stacks. Our results partially generalize  the work of Iwanari, Fantechi-Mann-Nironi, and Geraschenko-Satriano for abstract toric stacks.  
\end{abstract}

\maketitle
\tableofcontents

\section{Introduction}
 Several theories of \emph{abstract toric stacks}, i.e., algebraic stacks with a torus action and a dense open substack isomorphic to the torus, have been introduced over the last years; see  \cite{Laf02, BCS05,Iwa09,FMN,Tyo12,GS15I,GS15II,GM}. These stacks admit a simple combinatorial description, via   \emph{stacky fans}, and thus   provide a class of stacks which are easy to handle. Moreover, in some cases they have a natural interpretation in terms of moduli spaces (e.g. as the parameter space of certain tuples of effective Cartier divisors on toric varieties in \cite[Section 7]{GS15II}).  Also, certain toric stacks appear naturally as Mori dream   stacks; see   \cite{ElHo}.

The aim of this paper is to generalize some structure results from the setting of abstract toric stacks to the more general setting of \emph{abstract horospherical stacks}. More precisely, we characterize algebraic stacks with a reductive group action and a dense open substack isomorphic to a horospherical homogeneous space as stacky quotients of horospherical varieties. 
Let us mention that we were first led to investigate this problem by the work of Borisov-Chen-Smith \cite{BCS05}, Iwanari \cite{Iwa09}, Fantechi-Mann-Nironi \cite{FMN}, and Geraschenko-Satriano \cite{GS15I,GS15II} on abstract toric stacks. 

To state our   results, we first   review the basic definitions; see also Section \ref{section:defns}.
Let $k$ be an algebraically closed field  of characteristic zero, and let $G$ be a connected reductive linear algebraic group over $k$. 
A closed subgroup $H$ of $G$ is \emph{horospherical} if it contains a maximal unipotent subgroup of $G$. In this case,  the normalizer $P:= N_G(H)$ of $H$ in $G$  is a parabolic subgroup of $G$ and the quotient $\T:= P/H=\Aut^G(G/H)$ is a torus. A homogeneous space $G/H$ is \emph{horospherical} if $H$ is a horospherical subgroup of $G$. Note that    the natural morphism  $G/H\rightarrow G/P$ is a  Zariski $\T$-torsor over the flag variety $G/P$. 

A \emph{horospherical $G$-variety} $X$ is a normal $G$-variety with an open horospherical $G$-orbit; see for instance \cite{Pas06,Pas08} for a presentation of the theory of horospherical varieties (and their relation to Fano varieties). Horospherical varieties form a subclass of spherical varieties \cite{Pau81, Kno91, Per14} containing both toric varieties and flag varieties. The advantage to working with horospherical varieties is that their combinatorial description is easier than that of a general spherical variety.

Horospherical varieties appear naturally as orbit closures of certain linear representations \cite{PV72}. Moreover, they form a fertile ground to tackle some problems in algebraic geometry such as the Mukai conjecture \cite{Pas10}. The theory of horospherical varieties is also exploited in the (log) minimal model program  \cite{Pa18}, the study of stringy invariants \cite{BM13,LPR}, and quantum cohomology \cite{GPPS}. Thus, it seems reasonable to suspect that a ``stacky'' generalization of the notion of a horospherical variety could be useful in  algebraic geometry.

Simply replacing the word "variety" by the word "stack" in the definition of a horospherical variety, we obtain a first (naive) generalization of the notion of horospherical $G$-variety, as we explain now. We say that a finite type normal algebraic stack $\mathcal X$ over $k$ endowed with a $G$-action is an  \emph{abstract horospherical $G$-stack} if  there is a $G$-stable dense open substack of $\X$ which is $G$-isomorphic to a horospherical homogeneous space $G/H$. If, in addition, the rational  map $\X \dashrightarrow G/P$ induced by the open immersion $G/H \hookrightarrow \X$ is a morphism of stacks, then we say that $\X$ is a \emph{toroidal} abstract horospherical $G$-stack; see Definitions \ref{defn:embeddings} and   \ref{defn:toroidal_embeddings}. We note that  we recover the classical theory of horospherical varieties by considering stacks which are (representable by) varieties.  

Let us note that if $G=\T$ is a torus and $H=\{1\}$, then the (toroidal) abstract horospherical $G$-stacks with a dense open substack $G$-isomorphic to $G/H = \T$ are precisely the \emph{abstract toric stacks} considered in the work of Geraschenko-Satriano \cite{GS15I, GS15II}.
We recall the main result of \cite{GS15II} which is a characterization of certain abstract toric stacks as quotients of toric varieties. (We refer the reader to Remark \ref{remark:gs} for a brief discussion of a  mistake in Geraschenko--Satriano's paper \cite{GS15II} which does not concern the following result.)

\begin{theorem} \label{th GS}
\cite[Theorem 5.2]{GS15II} --- 
If $\X$ is a smooth abstract toric stack such that the diagonal of $\X$ is affine and the geometric points of $\X$ have reductive inertia groups, then $\X$ is equivariantly isomorphic to a quotient stack $[X/K]$, where $X$ is a toric variety with torus $T$ and $K$ is a closed subgroup of $T$. 
\end{theorem}

It is stressed in \cite[Section 5.1]{GS15II} that Theorem \ref{th GS} fails without the assumption on the diagonal and the inertia groups. Moreover, it is shown in \cite{GM} that Theorem \ref{th GS} fails if one drops the smoothness assumption.

Our main results are a characterization of certain abstract horospherical $G$-stacks, without the restriction that $G$ is a torus, as quotients of horospherical varieties; see Theorems \ref{thm1} and \ref{thm2} below.   

The definition of abstract horospherical stacks above is quite natural but it is not the most convenient one to work with in    concrete examples.
Therefore, we introduce another class of stacks that can be studied using combinatorial tools; our goal will be to show that these two definitions coincide, under suitable assumptions. Let $\X$ be an algebraic stack with a $G$-action over $k$. We say that $\X$ is a \emph{horospherical $G$-stack} if there exist a horospherical $G\times T$-variety $X$, where $T$ is a torus acting faithfully on $X$, and a closed subgroup $K$ of $\mathrm{Aut}^{G\times T}(X)$ containing $T$, such that $\X$ is $G$-isomorphic to the stacky quotient $[X/K]$; see Definition \ref{defn:ho_st}. 

We note that the class of horospherical stacks is an intermediate class between the class of horospherical varieties and the class of abstract horospherical stacks. Indeed, it suffices to take $T=K=\{1\}$ for $\X$ to be a horospherical $G$-variety. 

Our first result characterizes toroidal abstract horospherical stacks.

\begin{theorem}\label{thm1}
If $\X$ is a smooth toroidal abstract horospherical $G$-stack such that the diagonal of $\X$ is affine and the geometric points of $\X$ have reductive inertia groups, then $\X$ is a horospherical $G$-stack.
\end{theorem}

We push our methods a bit further and obtain a general structure result for smooth (not necessarily toroidal) abstract horospherical stacks, under suitable assumptions. Namely, to prove our main result, we require the following conjecture (see also Section \ref{section:towards}). We will say that an open substack $\Y$ in a stack $\X$ is \emph{big} if its complement has codimension at least  $2$.  Moreover,  a linear algebraic group is \textit{diagonalizable} if it is a subgroup of a torus.

\begin{conjecture}[Criterion for  {quasi-affineness}]\label{conj} Let $G$ be a connected reductive algebraic group.
Let $ \X$ be a smooth integral finite type algebraic stack over $k$ with affine diagonal, $\Pic(\X)=0$, and diagonalizable inertia groups. Suppose that $ \X$ contains a big open substack $\Y$. If $\Y$ is a (smooth) quasi-affine scheme and $\X$ is an abstract horospherical $G$-stack, then $\X$ is a quasi-affine scheme. 
\end{conjecture}

To prove our next result, we use   Theorem \ref{thm1}, the theory of Cox rings  of horospherical varieties, and Conjecture \ref{conj}.

\begin{theorem}\label{thm2}
 Assume that Conjecture \ref{conj} holds. 	If $\X$ is a smooth  abstract horospherical $G$-stack with dense open substack $G/H$ such that the diagonal of $\X$ is affine, the geometric points of $\X$ have reductive inertia groups, and the natural (right) action of the torus $\T=P/H$ on $G/H$ extends to $\X$, then $\X$ is a horospherical $G$-stack.
\end{theorem}

Using the techniques employed in the proof of Theorem \ref{thm2}, we also obtain the following result.

\begin{theorem}\label{thm3}
Let $\X$ be a smooth abstract horospherical $G$-stack such that the natural (right) action of the torus $\T=P/H$ on $G/H$ extends to $\X$. If $\X$ admits an open covering by horospherical $G$-stacks, then $\X$ is a horospherical $G$-stack. 
\end{theorem}

To prove Theorem \ref{thm3}, we first establish  that Conjecture \ref{conj} holds true for horospherical $G$-stacks (Proposition \ref{prop:conjquotientstack2}). Our motivation to prove Theorem \ref{thm3} comes  from the proof of Theorem \ref{th GS} (by Geraschenko-Satriano) for a smooth abstract toric stack $\X$. Indeed, the proof Theorem \ref{th GS} consists of  first showing the existence of an open covering of $\X$ by toric stacks \cite[Theorem 4.5]{GS15II} and then an analogue of Theorem \ref{thm3} in the setting of smooth abstract toric stacks.

In Section \ref{section:lemmas} we recall some properties of   group actions on algebraic stacks, and include several properties of the normalization of an algebraic stack (which might be of independent interest). Next, we   define the class of abstract horospherical stacks and the subclass of horospherical stacks in Section \ref{section:defns}, and prove that abstract horospherical stacks have diagonalizable inertia groups using Luna's \'etale slice theorem for algebraic stacks \cite{AHR}, under reasonable assumptions (see Proposition \ref{prop:inertia_groups_are_always_diag}). 
Our first result (Theorem \ref{thm1}) is proven in Section \ref{section:toroidal}; see Theorem  \ref{thm:structure_theorem_for_toroidal}. Our proof reduces Theorem \ref{thm1} to the main result of Geraschenko-Satriano (Theorem \ref{th GS}) on abstract toric stacks, and therefore also relies crucially on Luna's \'etale slice theorem for algebraic stacks as proven by Alper-Hall-Rydh \cite{AHR}.  

As an application of our abstract structure results we construct \emph{toroidifications} of abstract horospherical $G$-stacks (Proposition \ref{prop:tor}) and we prove that abstract horospherical $G$-stacks have finitely many $G$-orbits (Corollary \ref{cor: finite G orbits}). Moreover, we show that, if the $G$-orbits of $\X$ are of codimension at most 1, then $\X$ is a smooth horospherical $G$-stack; see Proposition \ref{prop:codimension_one_case} for a precise statement. We then use all the previous results together with facts on Cox rings of horospherical varieties and Conjecture \ref{conj} to prove Theorems \ref{thm2} and \ref{thm3} in Section \ref{section:towards}.

\begin{remark}
Throughout this paper we assume that the base field $k$ is algebraically closed of characteristic zero. The essential reason for this restriction on the characteristic is that our main results rely on the results of Geraschenko-Satriano \cite{GS15II} where the base field is assumed to be algebraically closed of characteristic zero.  
\end{remark}	

\begin{remark}
Part of our results could easily be extended to the setting of spherical varieties. However, in several places we use in a crucial way the particular features of horospherical varieties (e.g. to construct the toroidification in Proposition \ref{prop:tor} or to reduce to the toric case in the proof of Theorem \ref{thm1}). 
\end{remark}

\begin{remark}
Let us mention that Wedhorn considers  \emph{spherical spaces} in \cite{Wed}. These are families of spherical varieties over arbitrary base schemes. This is another generalization of the notion of a spherical variety which is different from ours since for us the base scheme is $\Spec k$. 
On the other hand, in \cite{Hau00} Hausen considers complex analytic spaces with a $G$-action and a dense open orbit $G$-isomorphic to a spherical homogeneous space. He then  obtains a criterion for algebraicity. This criterion  applies in our situation for abstract horospherical stacks which are algebraic spaces.  
\end{remark}

\noindent \textbf{Conventions.}
Throughout this article, we let $k$ be an algebraically closed field of characteristic zero. A variety (over $k$) is an integral separated finite type scheme over $k$. An algebraic group (over $k$) is a finite type group scheme over $k$. A linear algebraic group (over $k$) is an affine algebraic group. By a subgroup, we always mean an algebraic closed subgroup.
We use the conventions of the Stacks Project \cite[Tag 026N]{stacks-project} for algebraic stacks.

\section{Group actions on algebraic stacks}\label{section:lemmas}

In this section we gather presumably well-known properties of group actions on algebraic stacks.

\subsection{Stack-theoretic images}\label{section:stack_image}
Let $f:\X \to \Y$ be a morphism of finite type algebraic stacks over $k$. 
We define the \textit{stack-theoretic image} of $f:\X\to \Y$ to be the   closed substack $\Z$ of $\Y$ whose ideal is   the kernel of the natural morphism $\mathcal O_{\mathcal Y}\to f_\ast \mathcal{O}_\X$.  This  coincides with the scheme-theoretic image if $\X$ and $\Y$ are schemes  \cite[Tag 01R6]{stacks-project}.

Note that   $f$ factors through $\Z$. Moreover,  for any other closed substack $\Z'$ of $\Y$ such that $f$ factors through $\Z'$, we have $\Z \subseteq \Z'$.  By \cite[Tag 01R8]{stacks-project}, the induced morphism $\X \to \Z$   is dominant. Also, it follows from the minimality of $\Z$ that, if $\X$ is reduced, then $\Z$ is a reduced algebraic stack.   

\subsection{Orbits of group actions}  \label{section: group actions}
Let $\X$ be an algebraic stack over $k$, and let $G$ be a group scheme over $k$. We say that $\X$ is a \emph{$G$-stack} (over $k$) if  $G$ acts on the groupoid $\mathcal X$; see \cite[Definition~1.3.(i)]{Rom05}. Note that, if $G'\to G$ is a morphism of group schemes over $k$, then any $G$-stack naturally inherits the structure of a $G'$-stack. A morphism of $G$-stacks $\mathcal X\to \mathcal Y$ is the data of a morphism of $G$-groupoids $\mathcal X\to \mathcal Y$; see  \cite[Definition~1.3.(ii)]{Rom05}.

\begin{lemma}\label{lem:singular_locus}
	The singular locus of a $G$-stack is a $G$-stable closed substack.
\end{lemma}  
\begin{proof}
	Let $x$ be an object of a $G$-stack $\X$. Let $g$ be an element of $G$. Let $P\to \X$ be a presentation. Note that multiplication by $g$ induces an automorphism $g:\X\to\X$. Hence, pulling-back along $P\to \X$ induces an isomorphism $P'\to P$. Note that $P'\to \X$ is a presentation as well. If $p'\in P'$ is a point lying over $x'$ which maps to a point $p\in P$ lying over $x := gx'$, then $p'$ is regular if and only if $p$ is regular. Since regularity is local in the smooth topology, we conclude that $x$ is regular if and only if $x'$ is regular. This concludes the proof.  
\end{proof}

Let $m:G\times \mathcal X\to \mathcal X$ be a $G$-action on the algebraic stack $\X$ and let $x \in \X(k)$ be a $k$-point of $\mathcal X$. We denote by $\overline{G . x}$ the stack-theoretic image of the morphism $G\to \mathcal X$ obtained as the composition 
 \[ G=\xymatrix{ G\times_{k} \Spec k \ar[rr]^{\mathrm{id}_G \times x} & &  G\times_{k} \mathcal X \ar[rr]^{m} & & \mathcal X}.\] 
 
Note that, for all  $x$ in $\mathcal X(k)$,  the stack-theoretic image $\overline{G.x}$ of $x$ in $\mathcal X$ is a $G$-stable closed substack of $\mathcal X$.  Suppose that $G$ is irreducible. Then, as the morphism $G\to \overline{G . x}$ is dominant and $G$ is irreducible , it follows that $\overline{G . x}$ contains a dense irreducible constructible substack. Therefore, if $G$ is irreducible, then $\overline{G.x}$ is irreducible and reduced, hence integral.

Note that the codimension of $\overline{G.x}$ in $\X$ is well-defined.  Moreover, as the codimension of a substack in a finite type algebraic stack $\X$ over $k$ is bounded from above by the dimension of a smooth surjective presentation $R \to \X$,   the maximum of    $\mathrm{codim}(\overline{G.x},\X)$, as $x$ runs over $\X(k)$, is well-defined.  This will be used in Sections \ref{section: toroidification} and \ref{section:towards}.

For $x$ in $\X(k)$, we have a morphism

\[G\to \X\times_k \X, \quad g\mapsto (x,gx).\]
We pull-back the diagonal $\Delta:\X\to \X\times_k \X$ along this morphism and obtain a Cartesian diagram
\[
\xymatrix{
H \ar[rr] \ar[d] & & \X \ar[d]^{\Delta} \\
G \ar[rr] &  & \X\times_k \X	
}
\] We refer  to  $H$ as the \emph{stabilizer (group scheme)} of $x$. (We emphasize that if $\X$ is a $G$-stack and $x$ is an object of $\X(k)$, then  the inertia group of $x$ in $\X$ does not coincide with the stabilizer of $x$ in general.)
Note that $H$ is not necessarily a subgroup of $G$. However, $H$ is a group scheme over $k$. Indeed, if $S$ is a scheme over $k$, then the $S$-objects of $H$ are pairs $(g,a)$ with $g$ in $G(S)$ and $a:gx\to x$ an isomorphism in $\X(S)$ (where we consider $x$ as an object of $\X(S)$ via the functor $\X(k) \to \X(S)$). Now, let us define $(g,a). (g',a')$, where $(g,a)$ and $(g',a')$ are objects of $H(S)$. Define $g'':= g  g'$ in $G(S)$. Moreover, let $a'': gg'x \to gx\to x$ be defined as $a'$ multiplied with $g$ and composed with $a$. (Here we use that $G$ also acts on the morphisms in $X$.) Then, we define $(g,a) . (g',a') := (g'',a'')$. In particular, the morphism $H\to G$ given by $(g,a)\mapsto g$ is a homomorphism. Let $K$ be its image. The closed subscheme $K$ of $G$ is a subgroup scheme  and the morphism $H\to K$ is faithfully flat.

Note that $H$   acts on $G$  (via $H\to G$). This action is not necessarily free. Indeed, note that, if $(1,a)$ is an element of $H(k)$, then $a$ is an object of the inertia group $I_x$ of $x$. Now,  since $(g,a)g' = gg'$,   the element $(1,a)$   acts trivially for all $a$ in the inertia group of $x$. Conversely, any $a$ in $I_x$ gives an element $(1,a)$ of $H(k)$. We see that the kernel of the action of $H$ on $G$ is naturally the inertia group of $x$.

Now, as $H $ maps to $K$ equivariantly for the action on $G$,   there is a natural $G$-equivariant morphism of algebraic stacks 
\[ [G/H] \to G/K.\] 
 
We will refer to $[G/H]$ {as \emph{the (stack-theoretic structure of the) orbit of $x$ in $\X$}}.   Note that the morphism $G\to \overline{G . x}$ is $H$-invariant. Therefore, there is a natural   morphism  $[G/H]\to \overline{G . x}$. Since the morphism $G\to \overline{G . x}$ is dominant, the morphism $[G/H]\to \overline{G . x}$ is dominant. 
 We will say that the $G$-stack $\X$ has a finite number of $G$-orbits if the set (of $k$-isomorphism classes of objects of) $\X(k)$ is a finite union of $G$-orbits.

\subsection{Normalization and equivariant resolution of the indeterminacy locus}
 Let $\X$ be a finite type     algebraic stack over $k$. Our discussion of the normalization of $\X$ closely follows \cite[Appendix A]{AB}.
 A morphism of algebraic stacks   $\X'\to \X$ is a \textit{normalization (of $\X$)} if, for all smooth morphisms $U\to \X$ with $U$ a scheme, the 
 scheme $U\times_{\X} \X'$ is the normalization of $U$.
 
 If $X$ is an integral affine   scheme, say $X=\Spec A$, then the normalization of $X$ is given by 
 $\Spec B$ with $B$ the integral closure of $A$ in its field of fractions. More generally, if $\X$ is an algebraic stack, then one can construct a normalization morphism $\X'\to \X$ following the analogous construction for algebraic spaces given  in \cite[Tag 07U4]{stacks-project}. Moreover, by adapting the arguments in \cite[Tag 0BB4]{stacks-project} for algebraic spaces, it follows that $\X'$ is normal and unique up to unique isomorphism.
 Furthermore,  for all   normal integral algebraic stacks $\Y$ over $k$ and for 
 all dominant morphisms $\Y\to \X$ with $\X$ an integral algebraic stack  there is a   morphism $\Y\to \X'$ such that $\Y\to \X$ factors as $\Y\to \X'\to \X$.  Note that normal algebraic stacks behave like normal schemes in many ways. For instance, 	a finite type integral normal algebraic stack over $k$ is nonsingular in codimension one.
 
 \begin{lemma}\label{lem:finite_index}
 	Let $f:\X\to \Y$ be a quasi-finite   representable morphism of finite type algebraic stacks over $k$. Then, for all $x$ in $\X(k)$, the image of the inertia group $I_x$ in $I_{f(x)}$ is of finite index.
 \end{lemma}
 \begin{proof}
 	We follow the proof of \cite[Proposition 3.2]{GS15II}.  Let $G = I_{f(x)}$. Since $k$ is algebraically closed, the residual gerbe of $\mathcal Y$ at $f(x)$ is trivial. Therefore,  we have a stabilizer-preserving morphism $BG\to \mathcal Y$ \cite[Definition 2.10]{Alp10}. Since stabilizer-preserving morphisms are stable under base-change, it suffices to show that the morphism $BG \times_{\mathcal Y}\mathcal X\to BG$ induces finite index inclusions on inertia groups. 
 	
 	Let $\Spec k \to BG$ be the universal $G$-torsor over $BG$. Let $U = \Spec k\times_{BG} (BG\times_{\mathcal Y} \mathcal X)$ and note that $U$ is a $G$-torsor over $BG\times_{\mathcal Y} \mathcal X$. Moreover, since $f$ is quasi-finite and finite type, it follows that $U$ is quasi-finite and finite type over $\Spec k$. It follows that $U = \Spec A$, where $A$ is a zero-dimensional finite type $k$-algebra. 
 	
 	Let $H$ be the inertia group of a point in $BG\times_{\mathcal Y}\X$. Note that $H$ is the stabilizer of a point $u$ in $U$ (with respect  to the action of $G$). Therefore, the set $(G/H)(k)$ identifies with  a subset of $U(k)$. Since $U$ is finite over $k$, we conclude that $G/H$ is finite.
 	 \end{proof}

 \begin{lemma}\label{lem:normalization} Let $\X$ be a finite type integral algebraic stack over $k$.
 	The normalization   $\X'\to \X$ is a representable proper quasi-finite birational surjective morphism. Moreover, for all $x'$ in $\X'(k)$ with image $x$ in $\X(k)$, the image of the inertia group $I_{x'}$ in $I_x$ is of finite index. 
 \end{lemma}
 \begin{proof} Let $P\to \X$ be a smooth surjective morphism with $P$ a scheme.
 	Note that $P':=P\times_{\X} \X'$ is the normalization of $P$ and thus a scheme. Therefore, the normalization morphism  is representable by  \cite[Tag 04ZP]{stacks-project}.  Since, $P'\to P$ is finite surjective, it follows that $\X'\to \X$ is proper quasi-finite  and surjective; see \cite[Tag 02LA and Tag 02KV]{stacks-project}. To see that $\X'\to \X$ is birational, let $V'\subseteq P'$ be a dense open which is isomorphic (via $P'\to P$) to some dense open $V$ of $P$. The image $U'$ of $V'$ in $\X'$ is open and maps to the image $U$ of $V$ in $\X$. Since the morphism $U'\to U$ is an isomorphism after pull-back along the cover $V\to U$, the morphism $U'\to  U$ is an isomorphism, so that $\X'\to \X$ is birational. The last statement    follows from Lemma \ref{lem:finite_index}.  
 \end{proof}

\begin{remark}  
	The normalization morphism of an algebraic stack is not necessarily stabilizer-preserving. Indeed, let $C$ be the nodal cubic curve given by the equation $y^2 = x^3+x^2$ over $k$. Note that $C$ is stable with respect to the action of $\mu_2$ on $\mathbb A^2$ given by $(x,y)\mapsto (x,-y)$, and that the singular point $(0,0)$ of $C$ is fixed by this action. In particular, it defines a stacky point (with inertia group $\mu_2$) of the quotient stack $\mathcal{X}:=[C/\mu_2]$. Consider the action $t\mapsto -t$ of $\mu_2$ on $\mathbb A^1$. Now, the normalization morphism $\mathbb A^1\to C$ is given by $t\mapsto (t^2-1,t(t^2-1))$, and  it is $\mu_2$-equivariant. Moreover, the fibre over the singular point $(0,0)$ of $C$ consists of precisely two points: $1$ and $-1$. Their stabilizers are trivial. In particular, the corresponding inertia groups in the quotient stack $\mathcal{X}':=[\mathbb A^1/\mu_2]$ are trivial. This shows that the normalization morphism $\mathcal{X}'\to \X$ is not stabilizer-preserving. 
\end{remark}

\begin{remark} The normalization morphism of an algebraic stack does not preserve commutativity of the inertia groups.  Indeed, let $C$ be  $Z(x^2+y^2+z^2, x+y+z)$ in $\mathbb{A}^3$, and note that the symmetric group $S_3$ acts on $C$ by permuting the coordinates $x$, $y$, and $z$.  The stack $\mathcal{X} := [C/S_3]$ has precisely one stacky point. The inertia group of this stacky point is $S_3$ (hence non-abelian). However,   if $\mathcal{X}'\to \mathcal{X}$ is the normalization morphism, then $\X'$ has a unique stacky point, and the inertia group of this stacky point is $\mathbb{Z}/3\mathbb{Z}$.
\end{remark}

We now show that the indeterminacy locus of a rational map from an algebraic stack to a proper scheme can be resolved. 

\begin{proposition}\label{prop:diagram1}
	Let $\X$ be a normal algebraic $G$-stack of finite type over $k$ and let $\U$ be a $G$-stable dense open substack of $\X$. Let $Y$ be a  proper scheme over $k$ with a $G$-action. If  $\X\dashrightarrow Y$ is a $G$-equivariant rational map which is defined on $\U$, then  
  there exists a representable    proper birational surjective morphism of normal algebraic $G$-stacks $\X'\to \X$ which is an isomorphism over $\U$ such that the composed $G$-equivariant rational map $\X' \to\X \dashrightarrow   Y$ is defined everywhere. Moreover, $\X'\to \X$ induces finite index inclusions on inertia groups.
	 
\end{proposition}
\begin{proof} As $Y$ is a proper scheme over $k$, the morphism $\X\times_k Y\to \X$
obtained by base change is   proper and stabilizer-preserving.
Let $\Gamma \subseteq \mathcal U\times_k Y$ be the graph of the $G$-equivariant morphism $\mathcal U\to Y$.  We let $\overline{\Gamma}\subseteq \mathcal X \times_k Y$ be the closure of $\Gamma$ in $\mathcal X\times_k Y$. Moreover, let $\X'$ be the normalization of $\overline{\Gamma}$.   Note that the normalization map $\mathcal X' \to \overline{\Gamma}$ is a representable quasi-finite proper birational   morphism of $G$-stacks which induces finite index inclusions on inertia groups (Lemma \ref{lem:normalization}).  
 It follows that the composed morphism 
	\[
	\mathcal X' \to \overline{\Gamma} \to \mathcal X\times_k Y \to \X
	\] is  a representable   proper birational (surjective) morphism which induces finite index inclusions on inertia groups.
	Also, as the composed $G$-equivariant rational map $\X' \to\X \dashrightarrow   Y$ coincides over a dense open substack of $\X'$ with the composed morphism $\X' \rightarrow \X\times_{k} Y \rightarrow Y$, this concludes the proof of the proposition.
\end{proof}

\subsection{Algebraic stacks over homogeneous spaces}\label{section:par_ind}
Let $G$ be a linear algebraic group over $k$, let $P$ be a closed  subgroup of $G$, and let $\Y$ be a finite type $P$-stack. We define the algebraic stack $ G\times^P  \Y$ to be $(G\times \Y)/P$, where $P$ acts on $G\times \Y$ via 
$$ p \cdot  (g,y):=(gp^{-1},p \cdot  y) \;\; \text{ for all $p \in P$, $g \in G$, $y \in Y$}.$$
Note that $G$ acts on $G\times^P \Y$ by left multiplication on the first factor. Also, the projection $G\times \Y\to G$ induces a $G$-equivariant morphism $G\times^P\Y\to G/P$ such that the stack-theoretic  fibre over $P/P$ is isomorphic to $\Y$. Therefore, since $G$ acts transitively on $G/P$, all $k$-fibres of $G\times^P\Y\to G/P$   are  isomorphic to $\Y$.

\begin{proposition}\label{prop:ressayre} Let $\X$ be a $G$-stack, and let $\Y$ be a $P$-stack. 
	Let $\mathcal X\to G/P$ be a $G$-equivariant morphism whose fiber over $P/P$, equipped with the $P$-action induced by restricting the $G$-action, is the $P$-stack $\Y$. Then, $\X$ is $G$-isomorphic to $G\times^P \Y$ over $G/P$.
\end{proposition}
\begin{proof}
We follow the proof of Ressayre \cite[Lemme~6.1]{Res04}.  It suffices to show that the natural morphism $\pi:G\times \Y \to \X$ given by $(g,y)\mapsto g\cdot y$ is a $P$-torsor for the \'etale topology. To do so,  since $p:G\to G/P$ is an \'etale locally trivial $P$-torsor, it suffices to show that $G\times \Y\to \X$ has a section, locally for the \'etale topology.  To construct such a section,   let $\psi:\Omega\to G/P$ be an \'etale cover such that  $G\times_{G/P} \Omega$ is trivial over $\Omega$. Let $\sigma:\Omega\to G$ be such that $\psi = p\circ \sigma$. Define $U:= \Omega \times_{G/P} \X$. Note that the natural morphism $f:U\to \X$ is an \'etale cover.  Define $s:U\to G\times \Y$ by \[(\omega,x) \mapsto \left(\sigma(\omega), \sigma(\omega)^{-1}\cdot x\right).\]  Note that $\pi\circ s = f$. Thus, we conclude that $G\times \Y\to \X$ is a $P$-torsor. 
\end{proof}

\begin{corollary}\label{cor:para ind} 
Assume that $G$ is connected and $P$ is a parabolic subgroup of $G$.
If $\phi: \X \to G/P$ is a  	morphism of $G$-stacks and $\mathcal Y $ is the  (stack-theoretic) fiber of $\phi$ over $P/P$, then $\X\to G/P$ is a Zariski locally trivial  fibration which is   $G$-isomorphic to $G\times^P \Y$ over $G/P$.
\end{corollary}
\begin{proof}
	Since $P$ is a parabolic subgroup, the morphism $G\to G/P$ is a Zariski locally trivial $P$-torsor \cite[Theorem 4.13]{BT65}. In particular, the morphism $G\times^P \Y\to G/P$ is a Zariski locally trivial fibration whose fibers are isomorphic to $\mathcal Y$.  However, by Proposition \ref{prop:ressayre}, the algebraic stack $\X$ is $G$-isomorphic to $G\times^P\Y$ over $G/P$. We conclude that $\X$   is Zariski locally trivial over $G/P$. 
\end{proof}

\section{Abstract horospherical stacks}\label{section:defns}
In the following, we always denote by $G$ a connected  reductive linear algebraic group over $k$. 
We start by introducing the notion of abstract horospherical $G$-stacks. We keep the same notation as in the introduction. Namely, $G/H$ is a horospherical $G$-homogeneous space, $P$ denotes the normalizer of $H$ in $G$, and $\T$ is the torus $P/H=\Aut^G(G/H)$. 

In Section \ref{sec:defns} we give the   definition of an abstract horospherical $G$-stack' and in Section \ref{sec:iner} we show that the inertia groups of these stacks are diagonalizable using Luna's \'etale slice theorem for algebraic stacks \cite{AHR}.

\subsection{Definitions and basic properties} \label{sec:defns}
Our aim is to show that abstract horospherical $G$-stacks are quotient stacks (of a particular type). Our  definition of an abstract horospherical $G$-stack is as follows.
\begin{definition}\label{defn:embeddings} A normal integral algebraic $G$-stack $\X$ is an \emph{abstract horospherical $G$-stack} if $\X$ contains a $G$-stable dense open substack $G$-isomorphic to a horospherical homogeneous space $G/H$. 
\end{definition}

As mentioned in the introduction, we recover the classical theory of horospherical varieties by considering stacks which are (representable by) varieties.  

\begin{remark} \label{extension T action}
Let us note that if $X$ is a horospherical $G$-variety, then the $G$-equivariant automorphism group $\Aut^G(X)$ of $X$ coincides with the torus $\T = P/H=\Aut^G(G/H)$; see for instance \cite[Lemma 4.1]{AKP15}. In other words, the natural (right) action of $\T$ on the open orbit $G/H$ always extends to $X$. However we do not know whether the $\T$-action on $G/H$ always extends if we replace $X$ by an abstract horospherical $G$-stack.
\end{remark}

 \begin{definition}\label{defn:toroidal_embeddings} 
Let $\X$ be an abstract horospherical $G$-stack with dense open substack $G/H$.
  We say that $\X$ is \emph{toroidal} if the $G$-equivariant rational map $\X \dashrightarrow G/P$, induced by the open immersion $G/H \hookrightarrow \X$, is a morphism of $G$-stacks.
 \end{definition}

\begin{remark}
If $G$ is a torus and $H=\{1\}$, then  the (toroidal) abstract horospherical $G$-stacks are precisely the \emph{abstract toric stacks} studied by Geraschenko-Satriano in \cite{GS15I,GS15II}.
\end{remark}

 We now define horospherical $G$-stacks. We will see (Remark \ref{smoothness required}) that they form a proper subclass of the class of abstract horospherical $G$-stacks.

\begin{definition}\label{defn:ho_st} 
An algebraic $G$-stack $\X$ over $k$ is a \emph{horospherical $G$-stack} if there exist a horospherical $G\times T$-variety $X$, where $T$ is a torus acting faithfully on $X$, and a subgroup $K$ of $\mathrm{Aut}^{G\times T}(X)$ containing $T$, such that $\X$ is $G$-isomorphic to the stacky quotient $[X/K]$.
We will refer to $\X$ as the horospherical $G$-stack associated with the pair $(X,K)$. 
\end{definition}

\begin{remark}
If $G$ is a torus and $H=\{1\}$, then the  (toroidal) horospherical $G$-stacks are precisely the \emph{toric stacks} studied by Geraschenko-Satriano in \cite{GS15I,GS15II}, i.e., quotients  of a toric variety by a subgroup of the torus.
\end{remark}	

\begin{remark}
The diagonal of a horospherical $G$-stack $[X/K]$ is  affine (by \cite[Lemma~3.3]{GS15II}).
\end{remark}
\begin{remark} 
Let $\mathcal{X} = [X/K]$ be a horospherical $G$-stack. If $U$ denotes the open $G \times T$-orbit of $X$, then $U/K$ is a horospherical $G$-homogeneous space, and thus $[X/K]$ is an abstract horospherical $G$-stack in the sense of Definition \ref{defn:embeddings}.
\end{remark}
	
\begin{remark}
There is a well-developed combinatorial description for horospherical varieties; see for instance \cite{Pas06,Pas08}. Therefore, one could also obtain a combinatorial description for horospherical stacks proceeding as in \cite[Section 2]{GS15I} or \cite[Section 2]{GM}.
\end{remark}

\begin{example}
Let $G=\SL_2(k)$, $H=\begin{bmatrix}
1 & * \\
0 &  1
\end{bmatrix}$,  $P:=N_G(H)=\begin{bmatrix}
* & * \\
0 & *
\end{bmatrix}$,  and $\T:=P/H=\begin{bmatrix}
* & 0 \\
0 & *
\end{bmatrix} \cong \Gm$. The $G$-homogeneous space $G/H$ is horospherical and is isomorphic to $\A^2 \setminus \{0\}$ equipped with the natural action of $G$. It follows from the combinatorial description that the horospherical $G$-varieties with open orbit $G$-isomorphic to $G/H$ are the following: $\A^2 \setminus \{0\}$, $\A^2$, $\P^2$, $\P^2 \setminus \{0\}$, $\Bl_0(\A^2)$, and $\Bl_0(\P^2)$. If $K=\tilde K/H$ is any (closed) subgroup of $\T$ and $X$ is one of the six $G$-varieties above, then $\X=[X/K]$ is a horospherical $G$-stack with open orbit $G/\tilde K$. Also, if $Y$ is a toric variety with torus $T$, then $\X=[(X \times Y)/(\tilde K \times T)]$ is again a horospherical $G$-stack with open orbit $G/\tilde K$, but now $X \times Y$ is not a horospherical $G$-variety, it is a horospherical $G \times T$-variety.  
\end{example}

\begin{lemma}\label{lem:toroidal_basics}
Let $\mathcal X=[X/K]$ be a horospherical $G$-stack.
 Then the horospherical stack $\X$ is toroidal if and only if the horospherical variety $X$ is toroidal. 
\end{lemma}
\begin{proof}  
There is a commutative diagram 
\[	 \xymatrix{ X \ar[d]_{\textrm{quotient}} \ar@{-->}[rr] & & G'/P' = G/P \ar[d]^{\textrm{identity}} \\ \mathcal X = [X/K] \ar@{-->}[rr] & & G/P, } \] where $X$ is a $G' = G\times T$-horospherical variety and $\mathcal X$ is the horospherical $G$-stack associated with the pair $(X,K)$; see Definition \ref{defn:ho_st}. In particular, it is clear that, if $\mathcal X$ is toroidal, then $X$ is toroidal. Conversely, if $X$ is toroidal, then it follows that $\mathcal X$ is toroidal from the fact that the morphism $X\to G/P$ is $K$-invariant.
\end{proof}

\begin{lemma}\label{lem:algspace_is_scheme} \emph{(Horospherical algebraic spaces are schemes.)}\\
If $\X$ is a horospherical $G$-stack  and $\X$ is an algebraic space, then $\X$ is a scheme.
\end{lemma}
\begin{proof}
We follow the arguments in \cite[Remark~6.3]{GS15II}. Indeed, as $\X$ is a horospherical $G$-stack, there exist a horospherical $G'$-variety $X$ and a subgroup $K$ of $\T = \mathrm{Aut}^{G'}(X)$ such that $\X$ is $G$-isomorphic to $[X/K]$; see Definition \ref{defn:ho_st}. By Sumihiro's theorem \cite[Corollary 2]{Sum74},  the normal $\T$-variety $X$ is covered by $\T$-invariant (hence $K$-invariant) affine open subsets $U_{i}$. If $\X$ is an algebraic space, then $K$ acts freely on $X$, therefore for each $i$ the quotient $U_i \to U_{i}/\!\!/K = \Spec\,\O(U_{i})^{K}$ is a principal $K$-bundle (by Luna's slice theorem), and thus $U_{i}/\!\!/K$ coincides with $[U_{i}/K]$. This means that the algebraic space $\X$ admits an open covering by affine varieties and therefore is a scheme.
\end{proof}

\begin{remark}
Note that a horospherical stack which is an algebraic space might not be separated (e.g. the affine line with a double origin is a horospherical  $\Gm$-stack). 
\end{remark}

\subsection{Inertia groups of abstract horospherical stacks are diagonalizable}\label{sec:iner}
Recall that a linear algebraic group is \textit{diagonalizable} if it is a subgroup of a torus.  
The aim of this section is to apply Luna's \'etale slice theorem for algebraic stacks to prove that the inertia groups of abstract horospherical $G$-stacks with affine diagonal and reductive inertia groups are, in fact, diagonalizable; see Proposition \ref{prop:inertia_groups_are_always_diag}. We follow closely the line of reasoning used by Geraschenko-Satriano to prove \cite[Theorem~4.5]{GS15II}.

\begin{theorem}[Alper-Hall-Rydh]\label{thm:AHR}
Let $\mathcal{X}$ be a finite type integral algebraic stack with affine diagonal over $k$ whose geometric points have  reductive inertia groups. Let $x$ be a $k$-point of $\mathcal{X}$, and let $G_x$ be its inertia group. Then there exist an irreducible affine finite type scheme $Z$ over $k$ with an action of $G_x$, a $k$-point $w$ in $Z$ fixed by $G_x$, and a representable affine \'etale morphism 
\[
f:([Z/G_x],w)\to (\mathcal{X},x).
\]
\end{theorem}
\begin{proof}
Since affine morphisms of algebraic stacks are representable and finite type algebraic stacks over $k$ are quasi-separated \cite[Tag~01T7]{stacks-project}, the theorem follows from \cite[Theorem~1.2]{AHR}.
\end{proof}

\begin{lemma}\label{lem:inertia_gps_of_toric_embeddings}
Let $\mathcal{X}$ be a finite type integral (not necessarily normal) algebraic stack over $k$ with a dense open non-stacky $k$-point and affine diagonal. If     the geometric points of $\mathcal{X}$ have reductive inertia groups, then the inertia groups of $\mathcal{X}$ are tori.
\end{lemma}
\begin{proof} (We follow the proof of \cite[Theorem~4.5]{GS15II}, but replace the first paragraphs with a direct application of Luna's \'etale slice theorem, and avoid the last paragraph in \emph{loc. cit.}.) Let $x$ be a $k$-point of $\mathcal{X}$.  Let $G_x$ be the inertia group of $x$ (sometimes also referred to as the stabilizer of $x$). Since $G_x$ is reductive,  it follows from Alper-Hall-Rydh \cite[Theorem~1.2]{AHR} that there exist an irreducible affine finite type scheme $Z$ with an action of $G_x$, a $k$-point $w$ in $Z$ fixed by $G_x$, and a representable affine \'etale morphism 
\[
f: ([Z/G_x],w)\to (\mathcal{X},x)
\]
  such that $BG_x\cong f^{-1}(BG_x)$.  Since \'etale representable morphisms induce finite index inclusions on inertia groups \cite[Proposition~3.2]{GS15II}, and finite index subgroups of a linearly reductive group are linearly reductive, we see that the geometric points of $[Z/G_x]$ have linearly reductive inertia groups.  Since   
$\mathcal{X}$ contains a dense open (non-stacky) $k$-point and $[Z/G_x]\to \mathcal{Z}$ is \'etale representable, it follows that $[Z/G_x]$ contains a dense open non-stacky $k$-point. By reformulating the aforementioned properties, we see that the stabilizers for the action of $G_x$ on $Z$ are all reductive,  $w$ is a fixed point for this action, and $Z$ contains a open stabilizer-free orbit.  We conclude that $G_x$ is a torus \cite[Proposition~3.16]{GS15II}. 
\end{proof}

\begin{corollary}\label{cor:toric_emb_have_diag}
Let $T$ be a torus, and let $\mathcal{X}$ be a finite type integral algebraic stack over $k$ with an action of $T$. Assume that $\mathcal{X}$ has affine diagonal, contains a dense open stabilizer-free $T$-orbit,  and that every geometric point of $\X$ has a reductive inertia group. Then, every geometric point of $\mathcal{X}$ has a diagonalizable inertia group.
\end{corollary}
\begin{proof}
Since $\mathcal{X}\to[\mathcal{X}/T]$ is representable, it suffices to show that the inertia groups of $\mathcal{Y}:=[\mathcal{X}/T]$ are diagonalizable. However, by assumption, $\mathcal{Y}$ contains a dense open (non-stacky) $k$-point. Since $\mathcal{Y}$ is a finite type integral algebraic stack over $k$ with affine diagonal and reductive inertia groups, it follows from Lemma \ref{lem:inertia_gps_of_toric_embeddings} that $\mathcal{Y}$ has diagonalizable inertia groups.
\end{proof}
\begin{proposition}\label{prop:inertia_groups_are_always_diag}
Let $\mathcal{X}$ be a finite type integral (not necessarily normal) $G$-stack which contains a $G$-stable dense open substack $G$-isomorphic to (the horospherical homogeneous space) $G/H$. Then the inertia groups of $\mathcal{X}$ are diagonalizable groups.
\end{proposition}
\begin{proof}  
Let $\Gamma$ be the graph of the natural morphism $G/H\to G/P$. Let $\overline{\Gamma}$ be its closure in $\mathcal{X}\times G/P$. 
Now, $\overline{\Gamma}\to G/P$  is a $G$-equivariant morphism. Let $\mathcal{Y}$ be the fibre of $\overline{\Gamma}\to G/P$, and note that $P$ acts naturally on $\mathcal{Y}$. By Corollary \ref{cor:para ind}, it follows  that $\overline{\Gamma} = G\times^P \mathcal{Y}$. Now, $\mathcal{Y}$ inherits an action of $\mathbb{T}= P/H$ (as $H$ acts trivially), and $\mathcal{Y}$ has a dense open stabilizer-free orbit for the action of $\mathbb{T}$. Since $\mathbb{T}$ is a torus, it follows from Corollary \ref{cor:toric_emb_have_diag} that the geometric points of $\mathcal{Y}$ have diagonalizable inertia groups. (Here we use that $\mathcal{Y}$ has affine diagonal and reductive inertia groups.) In particular, $G\times \mathcal{Y}$ has diagonalizable inertia groups, and thus $G\times^P\mathcal{Y}$ has diagonalizable inertia groups. We conclude that $\overline{\Gamma}$ has diagonalizable inertia groups. 

Note that $\overline{\Gamma}$ is a   $G$-stable closed substack of $\mathcal{X}\times G/P$. In particular, as the natural projection $\mathcal{X}\times G/P\to \mathcal{X}$ is stabilizer-preserving, we conclude that the natural morphism $\overline{\Gamma}\to \mathcal{X}$ is stabilizer-preserving and $G$-equivariant.  Thus, $\mathcal{X}$ has diagonalizable inertia groups.
\end{proof}

\section{Describing toroidal abstract horospherical stacks}\label{section:toroidal}
The following lemma is an analog of Geraschenko-Satriano's \cite[Lemma~4.1]{GS15II} in the setting of abstract horospherical stacks.

\begin{lemma}  \label{lem: X horo if X/T horo}
Let $\X$ be an abstract horospherical $G$-stack with dense open substack $G/H$ and assume that the natural  (right) action of the torus $\T=P/H$ on $G/H$  extends to $\X$. Then, $\mathcal X$ is a horospherical $G$-stack if and only if $[\X/\T]$ is a horospherical $G$-stack.
\end{lemma}

\begin{proof}

Note that if $\X$ is a horospherical $G$-stack, then it is clear that $[\X/\T]$ is a horospherical $G$-stack.

We assume that $[\mathcal X/\T]$ is a horospherical $G$-stack. Therefore we can 
write $[\mathcal X/\T] = [X/K]$ with $X$ a  horospherical $G'$-variety and $K = \mathrm{Aut}^{G'}(X)$, where $G' = G\times T$ for some torus $T$ acting faithfully on $X$.  
 Denoting by $G'/H'$ the open $G'$-orbit of $X$,
the following diagram has Cartesian squares:
\begin{center}
	\begin{tikzpicture}[
	back line/.style={densely dotted},
	cross line/.style={preaction={draw=white, -,
			line width=6pt}}]
	\matrix (m) [matrix of math nodes,
	row sep=3em, column sep=3em,
	text height=1.5ex,
	text depth=0.25ex]{
		& G'/H' \times_{G/P} G/H & & G/H \\
		Z:=X \times_{[\X/\T]} \X & & \X \\
		& G'/H'=X\times_{[X/K]} G/P & & {G/P} \\
		X & & {[\X/\T]=[X/K]} \\
	};
	\path[->]
	(m-1-2) edge (m-1-4)
	edge (m-2-1)
	edge [back line] (m-3-2)
	(m-1-4) edge (m-3-4)
	edge (m-2-3)
	(m-2-1) edge [cross line] (m-2-3)
	edge (m-4-1)
	(m-3-2) edge [back line] (m-3-4)
	edge [back line] (m-4-1)
	(m-4-1) edge (m-4-3)
	(m-3-4) edge (m-4-3)
	(m-2-3) edge [cross line] (m-4-3);
	\end{tikzpicture}
\end{center}

In this diagram, the horizontal arrows from left to right are $K$-torsors. The vertical arrows from top to bottom are $\T$-torsors. The remaining arrows are equivariant open immersions. The group $G' \times \T$ acts transitively on $G'/H' \times_{G/P} G/H$, and the stabilizer $H''$ of the point $(eH',eH)$ contains $H' \times \{e\}$, that is, $G'/H' \times_{G/P} G/H= (G' \times \T)/H''$ is a horospherical $G' \times \T$-homogeneous space.

Since the algebraic $G' \times \T$-stack $Z$ is a $\T$-torsor over the variety $X$, it is an integral normal separated scheme of finite type.  {Thus,} 
 we conclude that $Z$ is   a horospherical $G' \times \T$-variety. This shows that $\mathcal X = [Z/K]$ is a horospherical $G$-stack.
\end{proof}

\begin{lemma}\label{lem:gen_toroidal}
Let $\X$ be a toroidal abstract horospherical $G$-stack with dense open substack $G/H$, and let $\T =  P/H$.  The following statements hold.
\begin{enumerate}
	\item There exist an integral normal algebraic  $\T$-stack $\mathcal Y$ of finite type over $k$ with a dense open substack which is $\T$-equivariantly isomorphic to $\T$ and an isomorphism of $G$-stacks $\mathcal X \cong G\times^P \mathcal Y$ over $G/P$, where $P$ acts on $\mathcal Y$ via $P\to \T$. 
	Moreover,
	the stack $\Y$ has affine diagonal and reductive inertia groups provided that $\X$ satisfies these properties.
	\item The group  $G \times \T$ acts on $\X \cong G\times^P \mathcal Y$ via  
	$$(g,pH). (g',y) := (gg',p.y)= (gg'p,y).$$
	\end{enumerate}
\end{lemma} 
\begin{proof}  
To prove $(1)$, note that the existence of an isomorphism of $G$-stacks $\mathcal X \cong G\times^P \mathcal Y$ over $G/P$ is Corollary \ref{cor:para ind} , where $\Y$ is the stack-theoretic fiber of $\X \to G/P$ over $P/P$. As the fibration $G \times^P \mathcal Y \to G/P$ is Zariski locally trivial, $\X$ is an integral normal finite type stack if and only if $\Y$ is an integral normal finite type stack. The last statement in $(1)$ also follows from the fact that the fibration $G\times^P \mathcal Y \to G/P$ is Zariski locally trivial.

 Since $\Y$ is an integral $P$-stack and the subgroup $H$ of $P$ acts trivially on the dense open substack $P/H$, the group $H$ acts trivially on $\Y$. Therefore, the $P$-stack  $\Y$ has a natural action of $\T =  P/H$. Also, the $\T$-stack $\mathcal Y$  contains a dense open substack (namely $P/H$) which is $\T$-equivariantly isomorphic to $\T$. 
 \end{proof}

In the proof of the next theorem we will use  Geraschenko-Satriano's local structure theorem for (not necessarily smooth) abstract toric stacks.

\begin{remark}\label{remark:gs} While the main result of Geraschenko--Satriano's paper (see \cite[Theorem~6.1]{GS15II}) is false without the smoothness assumption (see Gillam--Molcho \cite{GM} for a discussion and a counter-example), the ``local'' structure results for toric stacks in Geraschenko--Satriano's paper are correct, and follow mainly from a theorem of  Alper--Hall--Rydh \cite{AHR}. 
The mistake in Geraschenko--Satriano's paper appears in the proof of \cite[Theorem 2.13]{GS15II} which is used later in the proof of \cite[Theorem 6.1]{GS15II}. But all the other results in \cite{GS15II} do not depend on \cite[Theorem 6.1]{GS15II} and     remain valid.
\end{remark}

   \begin{theorem}\label{thm:structure_theorem_for_toroidal}
   	Let $\X$ be a toroidal abstract horospherical stack such that the diagonal of $\X$ is affine, and the geometric points of $\X$ have reductive inertia groups.   The following statements hold.
   	\begin{enumerate}
   		\item     There exist an integer $n\geq 1$ and an open  covering of $\X$ by horospherical $G$-substacks $\mathcal X_1, \ldots, \mathcal X_n$.
   		\item If $\X$ is smooth, then $\X$ is a horospherical $G$-stack.
   	\end{enumerate} 
   \end{theorem} 
   \begin{proof}
   	As before, we denote by $G/H$ the horospherical homogeneous space which identifies with a dense open substack in $\X$. By Lemma \ref{lem:gen_toroidal}, there exist an integral normal algebraic  $\T$-stack $\mathcal Y$ of finite type over $k$ with a dense open substack which is $\T$-equivariantly isomorphic to $\T$ and an isomorphism of $G$-stacks $\mathcal X \cong G\times^P \mathcal Y$ over $G/P$, where $P$ acts on $\mathcal Y$ via $P\to \T$. Moreover, since $\X$ has affine diagonal and reductive geometric inertia groups, it follows that the stack $\Y$ has affine diagonal and reductive geometric inertia groups.
   	   	
 It now follows from Geraschenko-Satriano's local structure results \cite[Lemma~4.1 and Theorem~4.5]{GS15II}  that $\Y$ is the union of open $\T$-substacks $\mathcal Y_1,\ldots, \mathcal Y_n$ such that $\mathcal Y_i=[Y_i/K_i]$, where $Y_i$ is a toric variety for the action of a torus $T_i$ and $K_i \subseteq T_i$.
 
  Let $\mathcal X_i = G\times^P \mathcal Y_i$, and note that $\mathcal X_i$ is an open $G$-substack of $\X$.  To conclude the proof of $(1)$, it suffices to show that $\mathcal X_i$ is a horospherical $G$-stack. By Lemma \ref{lem: X horo if X/T horo}, it suffices to show that $[\mathcal X_i/\T]$ is a horospherical $G$-stack. 
  To do so, note that the $G$-stack
 \[ 
 [\mathcal X_i/\T] \cong  [(G\times^P \Y_i)/\T] \cong  G/P\times [\Y_i/\T] \cong  G/P \times [Y_i/T_i] \cong [(G/P \times Y_i)/T_i] 
 \] is a horospherical $G$-stack.
   	
	 To prove $(2)$, we assume that $\mathcal X$ is smooth. In this case, the stack $\mathcal Y$ is smooth. Thus, it follows from Theorem \ref{th GS} that $\mathcal Y $ is a toric stack.  Write $\mathcal Y = [Y/K]$, where $Y$ is a toric variety with torus $T_Y$ and $K \subseteq T_Y$. Then, as before, $$[\mathcal X/\T] \cong G/P \times [Y/T_Y] \cong [ (G/P \times Y)/T_Y]$$ is a horospherical $G$-stack. Thus, by Lemma \ref{lem: X horo if X/T horo}, we conclude that $\mathcal X$ is a horospherical $G$-stack, as required.
\end{proof}

\begin{remark}  \label{smoothness required}
If one drops the smoothness assumption on $\X$ in the statement of Theorem \ref{thm:structure_theorem_for_toroidal}, then $\X$ is not necessarily a horospherical $G$-stack. See \cite[Section 4]{GM} for an example of a non-smooth abstract toric stack that is not a toric stack.
\end{remark}

\section{Toroidification}  \label{section: toroidification}

In this section, we apply the structure results obtained previously to construct the toroidification of an abstract horospherical $G$-stack $\X$ (Proposition \ref{prop:tor}) and we show that, if the $G$-orbits of $\X$ are of codimension at most $1$, then $\X$ is a smooth horospherical $G$-stack (Proposition \ref{prop:codimension_one_case}). This will be the starting point in our proof of Theorem \ref{thm2} in Section \ref{section:towards}.

\begin{proposition}[Toroidification] \label{prop:tor} 
Let $\mathcal X$ be an abstract horospherical $G$-stack with dense open substack $G/H$.
There exist a toroidal abstract horospherical $G$-stack $\X'$ and a representable proper  birational morphism of $G$-stacks $\X'\to \X$  which induces finite index inclusions on inertia groups.
\end{proposition}

\begin{proof} 
Recall that $P:=N_G(H)$ is a parabolic subgroup of $G$, i.e., the homogeneous space $G/P$ is a flag variety. 
Since $G/P$ is proper over $k$, by Proposition \ref{prop:diagram1}, there exists a representable proper  birational morphism of $G$-stacks $\X'\to \X$ which induces finite index inclusions on inertia groups and such that the induced rational  map $\X'\dashrightarrow G/P$ is a morphism.  In particular, the $G$-stack $\X'$ is  a toroidal abstract horospherical $G$-stack.  
\end{proof}

We refer to a morphism $\X'\to \X$ as in Proposition \ref{prop:tor} as a \emph{toroidification} of $\X$. Note that any toroidification induces an isomorphism over the dense open substack  $G/H$ of $ \mathcal{X}$.

\begin{remark}
There exists a natural choice of $\X'$ when $\X$ is a horospherical $G$-stack. 
Let $\X = [X/K]$ with $X$ a horospherical $G'$-variety and $G' = G\times T$. One can define $\X'$ as the horospherical $G$-stack $[X'/K]$, where $X'$
is the \emph{discoloration} of $X$;  {see \cite[Section 3.3]{Bri91} for an explicit construction of the discoloration of a spherical variety.}
\end{remark}

\begin{corollary}  \label{cor: finite G orbits}
	Let $\mathcal X$ be an abstract horospherical $G$-stack. If $\mathcal X$ has affine diagonal and reductive inertia groups, then $\mathcal X$ has only finitely many $G$-orbits.  
\end{corollary}
\begin{proof}
	Let $\X'\to \X$ be the toroidification morphism of Proposition \ref{prop:tor}. Since $\X$ has affine diagonal and reductive inertia groups, it follows that $\X'$ has affine diagonal and reductive inertia groups. Indeed, for all $x'$ in $\X'(k)$ with image $x$ in $\X(k)$, as the subgroup $\mathrm{Im}(I_{x'}\to I_x)$ is of finite index in $I_x$, we have that $I_{x'}^{0} = I_{x}^{0}$, and thus $I_{x'}$ is reductive. To see that the diagonal of $\X'$ is affine, it suffices to show that the diagonal of $\X'\to \X$ is affine. To do so, as the property of having affine diagonal is fppf local on the target, we may and do assume that $\X$ is a scheme. Since $\X'\to \X$ is representable, it follows that $\X'$ is an algebraic space. Since  $\X'\to \X$ is a proper morphism of algebraic spaces, it follows that the diagonal of $\X'\to \X$ is a closed immersion hence affine.   
	
	Now, as the toroidification morphism $\X'\to \X$ is $G$-equivariant and surjective, to prove the corollary, it suffices to show that $\X'$ has only finitely many $G$-orbits. Thus, we may and do assume that $\X = \X'$, so that $\X$ is a toroidal abstract horospherical $G$-stack.
	
	Now,  by  Theorem \ref{thm:structure_theorem_for_toroidal} $(1)$, there exist an integer $n\geq 1$ and an open covering of $\X$ by horospherical $G$-substacks $\X_1,\ldots, \X_n$. As the statement of the corollary is local on $\X$, we may and do assume that $n=1$, so that $\X$ is a toroidal horospherical $G$-stack.
	
	Finally, write  $\X =[X/K]$ with $X$  a toroidal horospherical $G'$-variety and $K$ a subgroup of its torus. Note that the $G$-orbits of $\X$ correspond one-to-one to the $G'$-orbits of $X$. Then, as $X$ has  only finitely many $G'$-orbits \cite[Theorem 2.1.2]{Per14}, it follows that $\X$ has only finitely many $G$-orbits.
\end{proof}

\begin{lemma}\label{lem:codimension}
	Let $\Y\to \X$ be a representable proper {birational morphism} of finite type integral algebraic stacks over  $k$. Let $D\subseteq \X$ be a closed integral substack, and let $D'$ be an irreducible component of its preimage which surjects onto $D$. Then $$\mathrm{codim}(D,\X) \geq \mathrm{codim}(D',\Y).$$
\end{lemma}
\begin{proof}
	Let $P\to \X$ be a smooth finite type surjective morphism with $P$ a scheme. As $\Y\to \X$ is representable, it follows that $\Y\times_\X P$ is an algebraic space. Since smooth finite type morphisms are codimension preserving, we may and do assume that $\Y$ and $\X$ are algebraic spaces in which case the statement of the lemma is well-known.
\end{proof}

\begin{proposition}\label{prop:codimension_one_case}
Let $\X$ be an abstract horospherical $G$-stack such that the diagonal of $\X$ is affine, and the geometric points of $\X$ have reductive inertia groups. Suppose that $\codim(\overline{G.x},\X) \leq 1$ for all $x$ in $\X(k)$. Then $\X$ is a  smooth toroidal horospherical stack.
\end{proposition}
\begin{proof}
Let $x$ be a singular object of $\X(k)$. Since the singular locus of $\X$ is a $G$-stable closed substack (Lemma \ref{lem:singular_locus}) and of codimension at least two (by the normality of $\X$), we see that   $\overline{G\cdot x}$ is of codimension at least two in $\X$. This contradicts our assumption that $\codim(\overline{G\cdot x},\X) \leq 1$ for all $x$ in $\X(k)$. It follows that $\X$ is smooth.

Let $f:\X'\to \X$ be a toroidification (Proposition \ref{prop:tor}). To conclude the proof, by Theorem \ref{thm:structure_theorem_for_toroidal},  it suffices to show that $f$ is an isomorphism.  

 As $f: \X'\to \X$ is a representable proper birational   morphism and $\X$ is an integral normal (even nonsingular) algebraic stack, it follows from Zariski's Main Theorem  \cite[Th\'eor\`eme~16.5]{LMB00}  that $\X'\to \X$ is an isomorphism provided that $f$ is quasi-finite. Thus, to conclude the proof, it suffices to show that $f$ is quasi-finite. 

To do so, 
let $x\in \X(k)$  be an object which is not a point of the open orbit $G/H$.  
Let $\mathcal Z$ be the inverse image of $\overline{G\cdot x}$ in $\X'$. Since   $\X'$ has only finitely many $G$-orbits (Corollary \ref{cor: finite G orbits}), there exist an integer $n\geq 1$ and objects $x_1',\ldots, x_n'$ in $\X'$ such that $\mathcal Z = \bigcup_{i=1}^n \overline{G\cdot x_i'}$.  

Note that $\overline{G\cdot x}$ is of codimension one in $\X$, by our assumption.  In particular, for all $i\in \{1,\ldots,n\}$, it follows that the closed substack $\overline{G\cdot x_i'}$ is of codimension one in $\X'$ (Lemma \ref{lem:codimension}). Therefore, pulling-back the latter morphism along a presentation of $\X$, a dimension argument (applied to the pull-back of $\X'\to\X$ along a  presentation of $\X$) shows that $\overline{G\cdot x_i'}\to \overline{G\cdot x}$  is generically quasi-finite.  (Here we only need that the codimension of $\overline{G\cdot x_i'}$ equals the codimension of $\overline{G\cdot x}$.)

Thus, for all $i\in \{1,\ldots, n\}$, there is a dense open $U_{i}$ of $\overline{G\cdot x}$ over which $\overline{G\cdot x_i'}\to \overline{G\cdot x}$ is quasi-finite. Let $U$ be the intersection of all $U_i$. Then, for all $i\in \{1,\ldots,n\}$, the morphism $\overline{G\cdot x'_i}\to \overline{G\cdot x}$ is quasi-finite over $U$. 

Since $U\subseteq \overline{G\cdot x}$ is a dense open, the union $V:=\bigcup_{g\in G} gU$ is a  $G$-stable dense open of $\overline{G\cdot x}$. Since $\X'\to \X$ is $G$-equivariant and quasi-finite over $U$, the morphism  $\X'\to \X$ is quasi-finite over $V$. We now show that $V= \overline{G\cdot x}$. To do so, we argue by contradiction.  Thus, let us assume that $V\neq \overline{G\cdot x}$. 

 Let $W$ be the complement of $V$ in $\overline{G\cdot x}$. Note that $W$ is a  {$G$-stable} closed substack of codimension at least one in $\overline{G\cdot x}$. In particular, $W$ is of codimension at least two in $\X$. Let $w$ be an object of  $W(k)$. Then  the closed substack $\overline{G\cdot w}$ is contained in $W$ and therefore is of codimension at least two in $\X$. This contradicts our assumption that $\codim(\overline{G\cdot x},\X) \leq 1$ for all $x$ in $\X(k)$.  Hence, $V = \overline{G\cdot x}$ and $f$ is quasi-finite.
\end{proof}

\section{Towards the general case}\label{section:towards} 
In this last section we first discuss Conjecture \ref{conj}. We prove Conjecture \ref{conj}, under suitable assumptions; see Lemma \ref{lem:conjquotientstack0} and Proposition \ref{prop:conjquotientstack2}. Finally, we use the theory of Cox rings and our results so far to  prove Theorems \ref{thm2} and \ref{thm3}.

\subsection{About Conjecture \ref{conj}}
We restate our conjecture for the reader's convenience.

\begin{conjecture*}[Criterion for quasi-affineness] Let $G$ be a connected reductive algebraic group.
Let $ \X$ be a smooth integral finite type algebraic stack over $k$ with affine diagonal, $\Pic(\X)=0$, and diagonalizable inertia groups. Suppose that $ \X$ contains a big open substack $\Y$. If $\Y$ is a (smooth) quasi-affine scheme and $\X$ is an abstract horospherical $G$-stack, then $  \X$ is a quasi-affine scheme. 
\end{conjecture*}

\begin{remark}\label{remark:schemes}
If $X$ is a smooth integral finite type scheme over $k$ with affine diagonal and trivial Picard group, then $X$ is  quasi-affine. Indeed, by \cite[Tag 01QE]{stacks-project} the scheme $X$ is quasi-affine if and only if its structure sheaf $\O_X$ is ample. To show that $\mathcal{O}_X$ is ample, we verify the conditions in \cite[Tag~01PS]{stacks-project}. To do so, let $x$ be a point of $X$, and let $U$ be an affine open subset of $X$ containing $x$. If $U=X$, then we are done. Thus, we may assume that $U\neq X$. The complement $D$ of $U$ in $X$ is pure of codimension one (as can be shown using \cite[Tag~0BCW]{stacks-project} and the fact that the diagonal of the smooth scheme $X$ over $k$ is affine). Let $s$ be a section of $\mathcal{O}_X(D)$ such that $\mathrm{div}(s) = D$.  Since $\Pic(X) = 0$, we see that $s$ is a section of $\mathcal{O}_X(D) \cong \mathcal{O}_X$.  Moreover, $X_s =U$ is affine and contains $x$. This shows that $\mathcal{O}_X$ is ample.
\end{remark}

 \begin{remark}
Let $X$ be the affine plane with a double origin over $k$. Then $X$ is a smooth finite type integral scheme over $k$ whose  diagonal is not affine. Indeed, $X$ contains two open subschemes $U$ and $V$ isomorphic to $\mathbb{A}^2_k$ whose intersection $U\cap V$ is isomorphic to $\mathbb{A}^2_k-\{0\}$. Moreover, $\mathrm{Pic}(X) = \mathrm{Cl}(X) =0$, and $X$ contains a big open subset $Y$ isomorphic to $\mathbb{A}^2_k-\{0\}$. Therefore, Conjecture \ref{conj} is false if one drops the hypothesis on the diagonal.
 \end{remark}

\begin{remark}[Kresch]\label{remark:kresch1} In Conjecture \ref{conj}, we can not remove the condition that $\X$ is an abstract horospherical $G$-stack.    Indeed, 
 let $F:=A_5$ be the alternating group, and let $F\to \mathrm{GL}(V)$ be the restriction of the standard linear representations of the symmetric group $S_5$.    Let $Y$  be the complement of the codimension two diagonals. (These diagonals are given by $x_i = x_j= x_k$ with $i,j,k$ pairwise distinct or by $x_i=x_j$ and $x_k=x_l$ with $i,j,k,l$ pairwise distinct in $V$.)  Note that $F$ acts stabilizer-free on the big open $Y$ of $V$. 

The stack $\mathcal{Z}:=[V/F]$ is not an algebraic space, as the origin in $V$ is a fixed point for the action of $F$. However, it is a smooth finite type separated Deligne-Mumford integral algebraic stack  with affine coarse space, $\Pic(\mathcal{Z})=0$, and reductive (finite) inertia groups. Moreover, the stack $\mathcal{Z}$ contains a big quasi-affine open $\Y := [Y/F]=Y/F$.  (To show that $\Pic(\mathcal{Z}) =0$, note that $F=A_5$ has no non-trivial characters, and  that the Picard group of $V$ is trivial. Since $V^F\neq \emptyset$ and $\mathcal{O}(V)^\times = k^\times$, it follows from    \cite[Corollary~5.3]{KKV89} that $\Pic(\mathcal{Z}) := \Pic_F(V) =0$.) 

Now, to give the desired example, let $W$ be the complement of the codimension three diagonals in $V$. (Note that the codimension three diagonals in $V=\mathbb{A}^5$ are given by $x_i= x_j =x_k =x_l$ with $i,j,k,l$ pairwise distinct or $x_i = x_j = x_k$ and $x_l=x_m$ with $i,j,k,l,m$ pairwise distinct.) The smooth finite type separated Deligne-Mumford stack $\X:=[W/F]$ has non-trivial inertia groups, and they are all finite abelian groups (either $\mathbb{Z}/2\mathbb{Z}$ or $\mathbb{Z}/3\mathbb{Z}$). Since $\X$ is not a scheme and contains a big quasi-affine open substack, this shows that we can not remove the condition that $\X$ is an abstract horospherical $G$-stack in Conjecture \ref{conj}. (Note that, as above, it follows from   \cite[Corollary~5.3]{KKV89} that $\Pic(\X) =0$.)
\end{remark}
 
  We now establish Conjecture \ref{conj} for certain stacks; see  Remark \ref{cor:first_case}, Corollary \ref{cor:second_case}, and Corollary \ref{cor:third_spec_case}.  To prove our results, we will frequently use the following well-known result.

\begin{lemma}\label{lem:surj_of_pic} Let $\X$ be a smooth finite type integral algebraic stack with quasi-compact and separated diagonal over $k$. 
Let $\Y$ be a dense open substack of $\X$.  Then, the natural homomorphism $\Pic(\X)\to \Pic(\Y)$ is surjective.
\end{lemma}
\begin{proof}
Let $L$ be a line bundle on $\Y$.  Let $i:\Y\to\X$ be the inclusion, and note that $i_\ast L$ is a quasi-coherent sheaf on $\X$   \cite[Proposition~13.2.6]{LMB00}. Since $L$ is   a coherent subsheaf of $i^\ast i_\ast L$, it follows from   \cite[Corollary~15.5]{LMB00} that there is a coherent sheaf $F$ on $\X$ such that $F|_{\Y} \cong L$. Let $M = F^{\vee\vee}$ be the double dual  of $F$. Note that $M|_{\Y} \cong L$. Let $P = \Spec A$ be a smooth   scheme over $k$ and let $f:P\to \X$ be a smooth surjective morphism. To conclude the proof,  it suffices to show that $f^\ast M$ is locally free. However, since taking duals is compatible  with flat pullback, the sheaf $f^\ast M$ is a reflexive coherent $\mathcal{O}_P$-module of rank $1$ (on the smooth $k$-scheme $P$), hence locally free \cite[Proposition~1.9]{Har80}. 
\end{proof}

\begin{lemma}\label{lem:conjquotientstack0}
Let $\X = [X/K]$ be a quotient stack where $K$ is a diagonalizable group acting
on a normal variety $X$.
Assume that  there exists a   big open substack $U\subseteq\X$ which is a quasi-affine scheme with $\Cl(U) =0$.
 If  $\O(X)^{\times}= k^{\times}$, 
then $K$ is trivial and  $\X=X$ is a scheme.
\end{lemma}
\begin{proof}
Let $\tau:X\rightarrow \X$ be the quotient map and consider the preimage $V:= \tau^{-1}(U)$. Note that the subscheme $V\subseteq X$ is a $K$-stable big open subset on which the $K$-action is free.   Let us show that $V$ is $K$-factorial. That is,  let us show that  every $K$-stable Weil divisor on $V$ is principal.

 Let $D$ be a $K$-stable Weil divisor on $V$. Since the map $\tau:V\rightarrow U$ is a $K$-torsor, there exists a Weil divisor $D'$ on $U$ such that $\tau^*(D') = D$. As $\Cl(U)=0$, there exists a rational function $f \in k(U)$ such that $\mathrm{div}_{U}(f) = D'$. 
It follows that $D=\tau^*\mathrm{div}_{U}(f)=\mathrm{div}_{V}(f \circ \tau)$ is a principal divisor. This  shows that $V$ is $K$-factorial.

Since $V\to U$ is a $K$-torsor and $K$ is affine over $k$, we see that $V$ is quasi-affine over $k$ (use \cite[Tag 02L5]{stacks-project}). Thus, $K$ is a diagonalizable group acting freely on the normal quasi-affine variety $V$, every invertible function on $V$ is constant (because $X$ is normal with $\O(X)^{\times}= k^{\times}$ and $V$ is a big open of $X$), and $V$ is $K$-factorial. Under these assumptions, by \cite[Proposition 2.7]{HS10}, it follows  that the character group $\chi(K)$ of $K$ is isomorphic to $\mathrm{Cl}(U)$. 
We conclude that $\chi(K) = \Cl(U) = 0$,  so that $K$ is trivial and $\X=X$ is a scheme. 
\end{proof}

\begin{remark}[First special case of Conjecture \ref{conj}]\label{cor:first_case}
If, in the notation of Lemma \ref{lem:conjquotientstack0},  we also assume that $\X$  is smooth, then $U$ is smooth and $\Cl(U)=\Pic(U)$ is trivial. Therefore, Lemma \ref{lem:conjquotientstack0} implies that Conjecture \ref{conj} holds for a quotient stack $[X/K]$ with $K$ a diagonalizable group and $X$ a smooth variety satisfying  $\O(X)^{\times}= k^{\times}$. 
\end{remark}

\begin{proposition}\label{prop:conjquotientstack2}
Let $\X = [X/K]$ be a smooth horospherical $G$-stack with $\Pic(\X) =0$ which contains a big open substack $\Y$. If $\Y$ is a quasi-affine scheme, then $\X$ is a horospherical variety. In other words, Conjecture \ref{conj} holds for smooth horospherical $G$-stacks.  
\end{proposition}

\begin{proof}
Let $\Y':=G \cdot \Y$, and note that $\Y'$  is a big open substack of $\X$ containing $\Y$. Since $\Pic(\X) =0$ and $\X$ has affine diagonal,     it follows from Lemma \ref{lem:surj_of_pic} that $\Pic(\Y')=\Pic(\Y)=0$. Thus $\Y'$ is a quasi-affine scheme by Remark \ref{remark:schemes}. Hence, replacing $\Y$ by $\Y'$ if necessary, we may and do assume that $\Y$ is $G$-stable. 

Let  $G'$ be the connected reductive group acting on $X$ with  dense open orbit $G'/H'$. By \cite[Theorem
2.1]{Kno91}, the $G'$-variety $X$ admits a $G'$-stable open covering $\{X_{i}\}$ such that any $X_{i}$ has a unique closed
$G'$-orbit. Thus, replacing $X$ by $X_i$ if necessary,  we may and do assume that $X$ has a unique closed $G'$-orbit.

By \cite[Theorem 2.3.2]{Per14}, there exist a parabolic
subgroup $P' \subseteq G'$ with unipotent radical $P_u'$, a $P'$-stable affine open subset $X_{0}\subseteq X$ with $X=G' \cdot X_0$, a closed $L$-stable subset $Z\subseteq X_{0}$ (where $L$ is
a Levi subgroup of $P'$), and an $L$-isomorphism
$$P'\times^{L}Z = P_{u}'\times Z\rightarrow X_{0},\, (u,x)\mapsto u\cdot x.$$ 
Note that $K\subseteq \mathbb T'$,  where $\mathbb{T}' = \mathrm{N}_{G'}(H')/H'$. Now, the complement of $X_{0}$ in $X$ is a union of  
$B'$-stable prime divisors intersecting $G'/H'$, where $B'\subseteq G'$ is a Borel subgroup contained in $P'$. In particular, $X_{0}$ is  
$\mathbb{T}'$-stable. Thus, since $G'\cdot X_0 = X$, to prove that $[X/K]$ is a scheme, it suffices to show that $[X_0/K]$ is a scheme.

As $\O(X_{0})^{P_u'} = \O(Z)$ and the $\T'$-action commutes with
the $P_{u}'$-action, we see that $Z$ is $\mathbb{T}'$-stable and
$\T'$ acts trivially on the first factor of $P_{u}'\times Z$.
Thus, since $[(P_u'\times Z)/K] = P_u'\times [Z/K]$, to prove that $[X_0/K]$ is a scheme, it suffices to show that $[Z/K]$ is a scheme. Since $Z$ is an affine variety, it suffices to show that $K$-acts freely on $Z$.

 {Since $Z$ is an $L$-horospherical variety \cite[Remark 2.3.3]{Per14}, each simple $L$-module appears with multiplicity at most one in the $L$-module $\O(Z)$. We denote by $\Lambda$ the corresponding set of dominant weights, and by $M$ the character group of $\mathbb T'$. The horosphericity of the $L$-action implies that the decomposition of $\O(Z)$ in simple $L$-modules
$$\O(Z) = \bigoplus_{\lambda \in \Lambda} V(\lambda)$$
is exactly the isotypic decomposition in $\T'$-modules. 
In particular, the lattice $\Lambda$ identifies with a sublattice of $M$.
Let $M_1 \subseteq \Lambda \subseteq M$ be the subset corresponding to the one-dimensional $L$-modules $V(\lambda)$ such that $\lambda \in \Lambda$ if and only if $\lambda^* \in \Lambda$. Since $Z$ is normal, the subset $M_1$ is a satured sublattice of $M$ and therefore $M = M_1 \oplus M_2$ for some sublattice $M_2 \subseteq M$. Then we can write $\O(Z)$ as a tensor product
 $$\O(Z) = \left( \bigoplus_{\lambda \in M_1} V(\lambda) \right) \otimes_k \left( \bigoplus_{\mu \in \Lambda \cap M_2} V(\mu) \right).$$
Since both sides are $\T' \times L$-algebras, there exist affine $\T' \times L$-varieties $Z_1$ and $Z_2$ such that $Z \cong Z_1 \times Z_2$ as a $\T' \times L$-variety. }

 {
The direct sum of lattices $M= M_1\oplus M_2$ induces a decomposition $\T' = \T_1 \times \T_2$, where each $\T_i$ acts on $Z_i$. (Note that $Z_1= \T_1$, where $\T_1$ acts on $Z_1$ by translation.) In particular, the $K$-action on $Z$ is induced by the product of the $K$-action on $Z_1$ and $Z_2$ via the inclusion $K\subseteq \T_1 \times \T_2$ composed with the projection $\T_1 \times \T_2 \to \T_i$ for $i=1,2$. }

 {
It follows from the definition of $\Lambda\cap M_2$ that $Z_2$ is an $L$-spherical variety with a unique fixed point by the $L$-action. As $X_0$ is smooth, $Z$ is smooth, and thus $Z_2$ is smooth. By   Luna's slice theorem (see \cite[\S III.1, Corollary 2]{Lun73}), the variety $Z_2$ is an $L$-module. Now, to prove that the $K$-action on $Z$ is free, it suffices to prove that the $K$-action on $Z_1$ is free. }

Let $\tau : X \rightarrow  \X=[X/K]$ be the natural quotient
map. The preimage $X_0' = X_0 \cap \tau^{-1}(\Y)$ is a $P'$-stable big open subset of $X_0$ on which $K$ acts freely. Moreover, the 
 quotient $X_0' /K$ is a scheme. We have the decomposition $X_{0}' = P_{u}' \times U_{0}$, where $U_{0}\subseteq Z$ is a big open subset on which $K$ acts freely. Let $V_0 := \T'\cdot U_0$ and note that $V_0$ is a $\T'$-stable big open of $Z$. Note that $[V_0 /K]$ is an algebraic space. The same argument as in the proof of Lemma \ref{lem:algspace_is_scheme} shows that $[V_0/K]$ is in fact a scheme.  As $\Cl(\X)=\Cl(\Y)=0$, it follows that $\Cl([(P_u'\times V_0)/K)=0$. 
 Therefore,
\[0=  \Cl([(P_u'\times V_0)/K]) = \Cl(P_u'\times [V_0/K]) =\Cl(P_u'\times (V_0/K))= \Cl(V_0/K)\]
since $P_u'$ is an affine space \cite[\S II, Proposition 6.6]{Har77}.  

Let us note that $\O(V_0)^{\times}=\O(Z)^{\times}=\O(Z_1)^{\times}$ since $Z_2$ is an affine space (see \cite[\S 1.1, Proposition]{KKV89}). By \cite[\S 5.1, Proposition]{KKV89}, there is an exact sequence
$$\left( \O(V_0)^{\times}/k^{\times} \right)^K= \left( \O(Z_{1})^{\times}/k^{\times} \right)^K \rightarrow \chi(K)\rightarrow \mathrm{H}^1_{\mathrm{alg}}(K,\O(V_0)^\times)=\Cl(V_0/K) = 0.$$
The surjectivity of the first map implies that the induced action of $K$ on $Z_1$ is faithful. Therefore, $K$ acts freely on an open subset of $Z_1$. Since $Z_1$ has a unique $\T'$-orbit, the action of $K$ on $Z_1$ is free everywhere. This concludes the proof of the proposition. 
\end{proof}

\begin{corollary}[Second special case of Conjecture \ref{conj}]\label{cor:second_case} Let $ \X$ be a smooth integral finite type algebraic stack over $k$ with affine diagonal, $\Pic(\X)=0$, and diagonalizable inertia groups. Suppose that $ \X$ contains a big open substack $\Y$. If $\Y$ is a (smooth) quasi-affine scheme and 
 $\mathcal{X}$ is Zariski covered by horospherical $G$-stacks, then $\mathcal{X}$ is a quasi-affine scheme.
\end{corollary}
\begin{proof}  
 Let $\X = \bigcup_i \X_i$ be a Zariski covering of $\X$ by horospherical $G$-stacks.  Also, note that $\Y_i:=\Y \cap \X_i$  is a big open quasi-affine substack of $\X_i$.  Since $\Pic(\X) =0$, it follows from Lemma \ref{lem:surj_of_pic} that  $\Pic(\Y_i)=\Pic(\X_i) =0$. Thus,  it follows from Proposition \ref{prop:conjquotientstack2} that $\X_i$ is a scheme. Thus, $\X$ is Zariski covered by schemes. Therefore,  $\X$ is a scheme. The result now follows from Remark \ref{remark:schemes}.
\end{proof}

\begin{lemma}\label{lem:chars_inject} Let $G$ be an algebraic group over $k$, and let $X$ be a  finite type scheme  with an action of $G$. Suppose that there is a point $x_0$ in $X(k)$ which is fixed by $G$. Then, 
 there is an injective group homomorphism from the character group $\mathbb{X}(G)$ of $G$ into $\Pic_G(X) = \Pic([X/G])$.
\end{lemma}
\begin{proof}
For a given $\chi \in \mathbb{X}(K)$, let $L_\chi:=X \times \A^1$ be the trivial line bundle equipped with the linearization $g.(x,v):=(g.x,\chi(g)v)$. Note that the map which sends $\chi$ in $\mathbb{X}(G)$ to $L_{\chi}$ in $\Pic([X/G])$ is   a group homomorphism. To show that this map is injective, assume that $L_{\chi}$ is isomorphic to the trivial line bundle with the trivial linearization. Then, for all $g$ in $G$, $x$ in $X(k)$ and $v$ in $\mathbb{A}^1$, we have  
$$(g.x,\chi(g)v)=(g.x,v).$$ In particular, as $x_0$ is a fixed point, we see that, for all $g$ in $G$ and all $v$ in $\mathbb{A}^1$, 
$$ (x_0,\chi(g) v) = (g.x_0,\chi(g) v) =(g.x_0,v) =(x_0,v).$$
We conclude that, for all $g$ in $G$, $\chi(g)  =1$, so that  $\chi$ is the trivial character.
\end{proof}

\begin{example}
Without the assumption that there exists a fixed point   $x_0$,     Lemma \ref{lem:chars_inject} might fail. Consider for example $G=\mu_2$, $X=\mathbb G_m$ with the usual (free) action, and note that $[\mathbb G_m/G] \cong \mathbb G_m$. Since $\Pic(\mathbb{G}_m)=0$ and $\chi(G) =\mathbb{Z}/2\mathbb{Z}$, there is no injective map $\chi(G)\to \Pic_G(\mathbb G_m)$. 
\end{example}

\begin{lemma}  \label{lem:conjecture_for_good_stacks}
Let $ \X$ be a smooth integral finite type algebraic stack over $k$ with affine diagonal, $\Pic(\X)=0$, and diagonalizable inertia groups. Suppose that $\X = [X/G]$, where $X$ is a  finite type scheme and $G$ is an algebraic group acting on $X$.  If $X$ has a fixed point, then $\X$ is a scheme. 
\end{lemma}
\begin{proof}
By Lemma \ref{lem:chars_inject}, $\mathbb{X}(G)$ injects into $\Pic(\X)=0$. Thus, $\mathbb{X}(G) =0$. Since $G$ is diagonalizable, we see that $G =0$. Thus, $\X = X$ is a scheme.
\end{proof}
 
 \begin{corollary}[Third special case of Conjecture \ref{conj}]\label{cor:third_spec_case}
Let $ \X$ be a smooth integral finite type algebraic stack over $k$ with affine diagonal, $\Pic(\X)=0$, and diagonalizable inertia groups. Suppose that $ \X$ contains a big open substack $\Y$ with $\Y$ a quasi-affine scheme.  Assume that there is  a Zariski open covering $\X = \bigcup_i \X_i$ with $\X_i = [U_i/G_i]$, where $U_i$ is a finite type scheme over $k$ and $G_i$ is an algebraic group over $k$ acting on $U_i$ with a fixed point. Then,  the stack $\X$ is a quasi-affine scheme.
\end{corollary}
\begin{proof}
Note that $\X_i$ contains a big open quasi-affine substack and that Lemma \ref{lem:surj_of_pic} implies that $\Pic(\X_i) =0$. Thus, by Lemma \ref{lem:conjecture_for_good_stacks}, the stack $\X_i$ is a scheme. Therefore, $\X$ is a scheme, so that $\X$ is quasi-affine (Remark \ref{remark:schemes}).
\end{proof}

\begin{remark}[Kresch] Let $F$ be a subgroup of $\mathbb{G}_m$ over $k$, and let $\X$ be a $F$-gerbe over a smooth finite type scheme $X$ over $k$ such that either the band of $\X$ is non-trivial, or the gerbe is trivially banded and its class in $\mathrm{H}^2(X,F)$ has nonzero image in $\mathrm{H}^2(X,\mathbb{G}_m)$. Then, $\X$ is a smooth finite type algebraic stack with affine diagonal which does not contain a dense open substack $[U/K]$ with $U$ a finite type scheme over $k$ and $K$ an algebraic group acting on $U$ with a fixed point. In particular, $\X$ does not satisfy the assumption of Corollary \ref{cor:third_spec_case}. For instance,  let $D_4$ be the dihedral group. Note that the center $Z$ of $D_4$ is $\mathbb{Z}/2\mathbb{Z}$, and let $D_4 \to V_4$ be the quotient map, where $V_4$ is the Klein four-group. Let $V_4 =<a,b>$ act on $V=\mathbb C^2$ via the reflections in the coordinate axis, i.e., $a.(x,y) = (-x,y)$ and $b.(x,y)=(x,-y)$. Note that $\mathcal{X} = [(\mathbb A^1\setminus \{0\})^2/D_4]$ is a non-trivial $\mu_2$-gerbe over $X:=(\mathbb A^1\setminus \{0\})^2$. Since $\mathcal{X}$ is a non-trivial $\mu_2$-gerbe, it is non-trivial on any dense open. Thus, there is no dense open substack $U$ of $\X$ such that $U = [V/\mu_2]$, where $V$ is a scheme and $\mu_2$ acts on $V$ with a fixed point.    (With the notation as in Remark \ref{remark:kresch1}, note that $\X = [W/A_5]$ also does not satisfy the assumption in Corollary \ref{cor:third_spec_case}. Otherwise, $\X$ would be a scheme by Corollary \ref{cor:third_spec_case}.)
\end{remark}

\begin{remark}\label{rem:bottom_up} Let $\mathcal{X}$ be a smooth finite type separated Deligne-Mumford algebraic stack with trivial generic stabilizer over $k$. Suppose that the coarse space of $X$ is a smooth scheme over $k$ and    that $\Pic(\mathcal{X})=0$. If $\mathcal{X}$ contains a big open substack $\mathcal{Y}$ which is quasi-affine,  then the   coarse space map $\mathcal{X}\to X$ is  an isomorphism by  Geraschenko--Satriano's ``bottom-up'' characterization of orbifolds \cite{GS17}. (Indeed,  the inertia groups of  codimension one points of $\mathcal{X}$ are trivial.) Thus, $\mathcal{X}$ is a smooth scheme and therefore quasi-affine  by Remark \ref{remark:schemes}. In particular, Conjecture \ref{conj} holds for smooth finite type separated Deligne-Mumford algebraic stacks with trivial generic stabilizer over $k$.
\end{remark}

 \subsection{Proof of Theorem \ref{thm2} and Theorem \ref{thm3}}

In our discussion below, we will require the following results.

\begin{proposition}\label{prop:extending} Let $\X$ be a smooth integral algebraic stack of finite type over $k$. Let $\Y$ be a big open substack of $\X$. Then the category of line bundles on $\Y$ is equivalent to the category of line bundles on $\X$.
\end{proposition}
\begin{proof} The essential surjectivity of this restriction functor follows from Lemma \ref{lem:surj_of_pic} (and doesn't require $\Y$ to be big). To prove the fully faithfulness, note that 
	if $\X$ is a scheme, then this is well-known as regular schemes are locally factorial; see \cite[Propositions 1.6 and 1.9]{Har80}. The  result for algebraic stacks follows from descent theory.  
\end{proof}

\begin{corollary}\label{cor:extending2} Let $X$ be a smooth integral scheme of finite type over $k$. Let $G$ be a linear algebraic group acting on $X$. Let $U$ be a $G$-stable big open subscheme of $X$. Then the category of $G$-linearized line bundles on $U$ is equivalent to the category of $G$-linearized line bundles on $X$.
\end{corollary}
\begin{proof}
	This follows from applying Proposition \ref{prop:extending} to $[U/G]\subseteq [X/G]$.
\end{proof}

Note that Theorems \ref{thm2} and \ref{thm3} are subsumed by the following result.
\begin{theorem}
 Let $\X$ be a smooth  abstract horospherical $G$-stack with dense open substack $G/H$ such that the diagonal of $\X$ is affine, and the geometric points of $\X$ have reductive inertia groups. Assume that the natural (right) action of the torus $\T=P/H$ on $G/H$ extends to $\X$. Then,
the following statements hold.
\begin{enumerate}
\item If  Conjecture \ref{conj} holds, then $\X$ is a horospherical $G$-stack.
\item If $\X$ has a Zariski-open covering by   horospherical $G$-stacks, then $\X$ is a horospherical $G$-stack.
\end{enumerate}
\end{theorem}
\begin{proof}   
Note that the (right) action of $\T=P/H$ on $G/H$ extends to the stack $\X$. (This is  part of our  assumption.)
Thus, we can apply  Lemma \ref{lem: X horo if X/T horo} and replace $\X$ by $[\X/\mathbb{T}]$. That is, we may and do assume that $\mathbb{T}=P/H $ is trivial, i.e., $H=P$. Moreover, replacing $G$ by a finite \'etale cover if necessary, we may and do assume that $G$ is a direct product of a torus and a simply-connected semisimple group.

Suppose that $\codim(\overline{G. x},\X)$ is at most $1$ for all $x$ in $\X(k)$. Then, the result follows from Proposition \ref{prop:codimension_one_case}. Therefore, to prove the theorem, we may and do assume that there exists $x \in \X(k)$ such that $\codim(\overline{G. x},\X)$ is at least $2$.

Define $\Z$ as the union of all closed substacks $\overline{G . x}$, where $x$ in $\X(k)$ runs over all points such that $\overline{G . x}$ is of codimension at least $2$. It follows from Corollary \ref{cor: finite G orbits} that $\Z$ is a $G$-stable closed substack of codimension at least $2$ in $\X$.

Let $\Y$ be the complement of $\Z$ in $\X$. Then, $\Y$ is a $G$-stable dense open substack of $\X$ such that $\codim(\overline{G.y},\Y)$ is at most $1$ for all $y \in \Y(k)$. Note that $\Y$ is a smooth abstract horospherical $G$-stack with dense open substack $G/H$. Therefore, by Proposition \ref{prop:codimension_one_case}, the stack $\Y$ is a smooth horospherical $G$-stack, i.e.,  there exist a smooth horospherical $G'$-variety $Y_1$ with open orbit $G'/H'$ and a diagonalizable subgroup $K_1 \subseteq \T_1:=P'/H'$ such that $\Y \cong [Y_1/K_1]$, where $G'=G \times T$ for some torus $T$ acting faithfully on $Y_1$ and contained in $K_1$. Since the open $G$-orbit of $\X$ is the flag variety $G/P$ (as $\mathbb T $ is trivial), we must have $K_1=\T_1$. 

 (We will now perform the Cox construction on $Y_1$; see \cite{ADHL15} for a general background on Cox rings. However, this is slightly more complicated than expected, as one needs to reduce to the situation in which there are only constant invertible global regular functions. Once this is done,  we will extend   the torsors over $\Y$ appearing below to $\X$ using that the complement is of codimension at least two. We will  then be in the situation of Conjecture \ref{conj}.)

	 Let $Y_2$ be a $G'$-equivariant smooth compactification of $Y_1$.  Then $\mathcal{O}(Y_2)^\times=k^\times$, and  $\Cl(Y_2)=\Pic(Y_2)$ is finitely generated and torsion-free; see \cite[Corollaries 3.2.5 and 3.2.6]{Per14}. Let $R(Y_2)$ be the Cox ring of $Y_2$, and let $Y_4:=\Spec R(Y_2)$. Note that $Y_4$ is  a normal (possibly singular) affine variety with trivial class group   \cite[Section 3.1]{Bri07}. The variety $Y_4$ comes with an action of a torus $\T_2$ whose group of characters identifies with $\Cl(Y_2)$, and there exists a $\T_2$-stable smooth quasi-affine subvariety $Y_3 \subseteq Y_4$ such that the $\T_2$-action on $Y_3$ is free and the quotient morphism $Y_3 \to Y_3/\T_2=Y_2$ is a $\T_2$-torsor. Moreover, the action of $G'$ on $Y_2$ lifts to $Y_3$ and $Y_4$ \cite[Theorem 4.2.3.2 and Section 4.5.4]{ADHL15}. (Here we use that $G$ is a direct product of a torus and a simply-connected semisimple   algebraic group, and that the class group of $Y_2$ is torsion-free.) Therefore $Y_3$ and $Y_4$ are horospherical $G''$-varieties, where $G''=G' \times \T_2=G \times T \times \T_2$, with open orbit $G''/H''$, say.
	 
	 We denote by $Y$ the inverse image of $Y_1$ in $Y_4$. Note that $Y$    is a smooth quasi-affine horospherical $G''$-variety. 
 
	  Let $\T_3=\T_2 \times \T_1$. The variety $Y$ is the total space of a $\T_3$-torsor over the stack $\mathcal Y=[Y/\T_3]$. Since $\mathrm{H}^1_{\mathrm{fppf}} (\Y, \Gm) = \Pic(\Y)$, the $\T_3$-torsor $Y$ corresponds to a direct sum of line bundles on $\Y$. As the codimension of $\Z=\X \setminus \Y$ in $\X$ is at least 2, such line bundles extend uniquely to $\X$ (Proposition \ref{prop:extending}). Therefore, as $\mathrm{H}^1_{\mathrm{fppf}} (\X, \Gm) = \Pic(\X)$, there exists a unique $\T_3$-torsor $X \to \X$ whose restriction over $\Y$ is $Y \to \Y$. As  $X$ is a  $\T_3$-torsor over $\X$, we see that $X$ is a smooth integral finite type algebraic stack with has affine diagonal and reductive inertia groups. Moreover, $Y$ is a smooth quasi-affine scheme with trivial Picard group and the complement of $Y$ in $X$ is of codimension at least $2$. Since $Y$ is a big open of $X$ with trivial Picard group, we see that $\Pic(X) =0$ (Proposition \ref{prop:extending}). 

Note that $G''$ acts on $Y$, and that that $G''$ acts on $\Y$ and $\X$ via the projection $G''\to G$. We now show that the  action of $G''$ on $Y$ extends to an action on $X$, so that $X$ is an abstract horospherical $G''$-stack.  We can view the $\T_3$-torsor $Y\to \Y$ as a direct sum of  $G''$-linearized line bundles on $\Y$. 
 However,      as $[\Y/G'']$ is a big dense open of $[\X/G'']$, it follows from Corollary \ref{cor:extending2}  that any $G''$-linearized line bundle on $\Y$ extends uniquely to a $G''$-linearized line bundle on $\X$. Therefore, the extension $X\to \X$ of the $\T_3$-torsor $Y \to \Y$ admits a compatible action of $G''$. Thus,  $X$ is an abstract horospherical $G''$-stack. In particular, every geometric point of  $X$ has a diagonalizable inertia group (Proposition \ref{prop:inertia_groups_are_always_diag}).

	The following diagram (whose squares are Cartesian) summarizes the situation so far:
	$$\xymatrix{
		Y_4=\Spec R(Y_2) & Y_3 \ar@{_{(}->}[l] \ar@{->>}[d]^{/\T_2} & Y \ar@{_{(}->}[l] \ar@{->>}[d]^{/\T_2}  \ar@{^{(}->}[r] & X \ar@{->>}[dd]^{/\T_2 \times \T_1} \\
		& [Y_3/\T_2]=Y_2  & Y_1 \ar@{_{(}->}[l]  \ar@{->>}[d]^{/\T_1} &  \\
		&  & \Y=[Y_1/\T_1] \ar@{^{(}->}[r] & \X
	}$$

Now, to prove $(1)$, note that  it follows from Conjecture \ref{conj} that $X$ is a quasi-affine scheme, and thus a smooth variety.  To prove $(2)$, note that $X$ is covered by horospherical $G$-stacks (as $\X$ is covered by horospherical $G$-stacks), so that $X$ is a smooth quasi-affine variety by Corollary \ref{cor:second_case}.
	
Hence, in both cases $(1)$ and $(2)$, we see that $X$ is a smooth horospherical $G''$-variety. Since  $\X = [X/\mathbb T_3]$, we see that $\X$ is a horospherical $G$-stack.
	\end{proof}
 
 \bigskip
 
 \noindent \textbf{Acknowledgements.}  We thank Matthieu Romagny for his help in writing Section \ref{section: group actions}. We thank Jarod Alper for his patience and numerous explanations, and his help in proving  Theorem \ref{thm2}.   We are grateful to Andrew Kresch for helpful discussions, and for providing us with several    examples (see Remark \ref{remark:kresch1}).
We thank Michel Brion, Adrien Dubouloz, Jack Hall, Andreas Hochenegger, Johannes Hofscheier, Elena Martinengo, Siddharth Mathur,  Martin Olsson, Matthew Satriano, David Rydh, and Angelo Vistoli for helpful and interesting discussions. We thank the referee for the helpful comments and for Remark \ref{rem:bottom_up}.
The first named author gratefully acknowledges support from SFB/Transregio 45.
The second named author is supported by the Heinrich Heine University. 
The research of the second author was conducted in the framework of the research training group GRK 2240: Algebro-geometric Methods in Algebra, Arithmetic and Topology, which is funded by DFG.
The  third named author is grateful to the Max-Planck-Institut f\"ur Mathematik of Bonn for the warm hospitality and support provided during the beginning of the writing of this paper. 
The authors are grateful to the Australian National University in  Canberra for its hospitality in March 2016 where part of this work was done.

\bibliographystyle{alpha}
\def\cprime{$'$}

\end{document}